\documentclass{amsart}
\usepackage[T1]{fontenc}
\usepackage[utf8x]{inputenc}

\usepackage{tikz}
\usepackage{graphicx}
\usepackage{booktabs}
\usepackage{amsmath, amssymb}
\usepackage{float}
\usepackage{subcaption}
\usepackage{caption}
\usepackage{xcolor}
\usepackage{listings}
\usepackage{booktabs}
\usepackage{placeins}   
\usepackage{comment}
\usepackage{adjustbox}

\usepackage{tcolorbox}
\tcbuselibrary{skins,breakable}

\usepackage{lmodern}

\usepackage{listings}
\newcommand{\algtitle}[1]{#1}

\usepackage{amscd,amsmath,amssymb,amsthm,amsfonts,epsfig,graphics}
\usepackage{graphicx}
\usepackage{epstopdf}
\usepackage[round,authoryear]{natbib}

\usepackage{xcolor}

\usepackage[colorlinks=true,citecolor=blue]{hyperref}
\usepackage{mathabx} 
\usepackage{hyperref}
\usepackage{subcaption}
\usepackage{comment}
\let\oldtocsection=\tocsection
\let\oldtocsubsection=\tocsubsection
\renewcommand{\tocsection}[2]{\hspace{0em}\oldtocsection{#1}{#2}}
\renewcommand{\tocsubsection}[2]{\hspace{1.8em}\oldtocsubsection{#1}{#2}}

\usepackage{mdframed} 
\usepackage{lipsum}
\usepackage{tcolorbox} 
\usepackage{comment}  

\tcbuselibrary{theorems}

\usepackage[margin=3cm]{geometry}
\usepackage{algorithm}
\usepackage{algpseudocode}
\newtcolorbox{blueheadgreenbox}[1]{%
  enhanced, breakable,
  colback=green!7,            
  colframe=green!50!black,    
  boxrule=0.8pt, arc=1mm,
  left=1.2ex, right=1.2ex, top=1.2ex, bottom=1.2ex,
  colbacktitle=blue!70!black, 
  coltitle=white,             
  fonttitle=\bfseries,
  boxed title style={sharp corners, boxrule=0pt},
  attach boxed title to top left={xshift=2mm,yshift*=-2mm},
  title={\algtitle{#1}}}
\newtheorem{theorem}{Theorem}[section]
\vspace{10pt}
\newtheorem{proposition}[theorem]{Proposition}
\newtheorem{lemma}[theorem]{Lemma}

\theoremstyle{definition}
\newtheorem{definition}[theorem]{Definition}

\vspace{10pt}
\newtheorem{assumption}[theorem]{Assumption}

\theoremstyle{remark}
\newtheorem{remark}[theorem]{Remark}
\newtheorem{corollary}[theorem]{Corollary}

\newtcbtheorem{mytheo}{Theorem}%
{colback=green!5,colframe=black!35!black,fonttitle=\bfseries}{th}

\newcommand{\mubar}{\bar{\mu}}

\begin{document}
\pagenumbering{arabic}

\title{Sharp Support Thresholds for Smeariness of Absolutely Continuous Measures on Spheres}
\vspace{-8mm}
\maketitle
\begin{center}
  \textbf{Susovan Pal} \\
  Department of Mathematics and Data Science, Vrije Universiteit Brussel (VUB) \\
  Pleinlaan 2, B-1050 Elsene/Ixelles, Belgium \\
  \texttt{susovan.pal@vub.be, susovan97@gmail.com}

\end{center}



\tableofcontents

\begin{abstract}
We investigate support thresholds for fully smeary and directionally smeary absolutely continuous probability measures on the sphere \(\mathbb{S}^m\). The motivation is inferential: smeariness is caused by degeneracy of the Hessian of the Fr\'echet function, and such degeneracy can invalidate the classical central limit theorem (CLT) for Fr\'echet means and the corresponding Wald-type \(\chi^2\) inference.

For rotationally symmetric densities, we show that full and directional smeariness are equivalent. The Hessian and fourth-order terms are governed by two explicit geometry-dependent radii \(R_m<S_m\). In dimensions \(m=2,3\), rotationally symmetric smeariness cannot occur. For \(m\ge4\), support contained in the geodesic ball of radius \(S_m\), centered at the Fr\'echet mean, rules out smeariness; conversely, for every \(\varepsilon>0\) with \(S_m+\varepsilon<\pi\), we construct examples of rotationally symmetric \(2\)-smeary densities supported in the ball of radius \(S_m+\varepsilon\).

For general densities, closed hemispherical support rules out both full and directional smeariness. Support contained in the closed ball of radius \(S_m\) rules out full smeariness, while we construct explicit, directionally \(2\)-smeary examples supported in balls of radius \(\pi/2+\varepsilon\).

As a byproduct, the explicit Hessian formulas in this paper also provide a practical diagnostic for detecting proximity to the Hessian-degenerate, non-classical regime.
\end{abstract}

\section{Introduction}\label{scn:introduction}

The Fr\'echet mean is the standard nonlinear analogue of the Euclidean expectation.  If \(X\) is a random variable on a Riemannian manifold \((M,g)\), with Riemannian distance \(d_g\), its \textit{population} Fr\'echet function is \(F_\mu(p):=\mathbb E[d_g^2(p,X)]\), \(p\in M\), and any minimizer of \(F_\mu\) is called a \textit{population} Fr\'echet mean.  Given i.i.d.\ random variables \(X_1,\ldots,X_n\sim X\), the\textit{ empirical} or \textit{sample} Fr\'echet function is \(F_n(p):=n^{-1}\sum_{j=1}^n d_g^2(p,X_j)\), and any minimizer of \(F_n\) is called a \textit{sample} Fr\'echet mean.  A measurable selection of sample Fr\'echet means will be denoted by \(\hat\mu_n\). Often, the word \textit{population} is dropped if the context is clear. Both \textit{local} and \textit{global} minimizers are used, and they are specified as local/global population/sample Fr\'echet means. Throughout this paper, the support-threshold results are local statements at the specified Fr\'echet mean: the hypotheses state when this point is assumed to be a unique local Fr\'echet mean, and global minimality is not asserted unless explicitly stated.

A central problem in statistics on manifolds is to understand when \(\hat\mu_n\) satisfies the classical \(n^{-1/2}\)-central limit theorem and when the usual Hessian-based normal approximation breaks down; see, for instance, \citet{BP,BP2,BL,HH,H_Procrustes_10}.  If the Hessian of $F_\mu$ at $\mubar$ is nondegenerate, the classical Bhattacharya--Patrangenaru central limit theorem (BPCLT for short) gives a Gaussian limit in the tangent space \(T_{\mubar}M\), with covariance involving the inverse Hessian of \(F_\mu\) at \(\mubar\).  This is the regime in which Hessian-corrected Wald-type procedures are meaningful.  If, however, the Hessian degenerates, the classical approximation can fail and slower, non-Gaussian behavior may occur; this phenomenon is known as\textit{ smeariness} (cf. Definition \ref{def:smeariness-directional}), following \citet{HotzHuckemann2015,HE}.  Thus, before applying normal or chi-square approximations for Fr\'echet means, it is natural to ask for geometric conditions which rule out Hessian degeneracy.

This paper answers that question on the unit sphere \(\mathbb S^m\) from the viewpoint of \textit{support}. More precisely, we ask: how far must the support of an absolutely continuous probability measure on \(\mathbb S^m\) extend from a prescribed unique local Fr\'echet mean before full or directional smeariness can occur? The answer is given by two support-threshold theorems. Under \textit{rotational symmetry}, there is a sharp threshold picture governed by two explicit geometry-dependent radii \(R_m<S_m\). For \textit{general} densities, closed hemispherical support is the sharp obstruction to directional smeariness, while the larger radius \(S_m\) is the corresponding obstruction for full smeariness. These results also give a population-level diagnostic for inference: if the support is known, or empirically appears, to lie safely below the relevant threshold, then one has evidence for the non-smeary Hessian regime. In practice, the sample radii \(d_g(\hat\mu_n,X_i)\) around the empirical Fr\'echet mean \(\hat\mu_n\) provide a natural support diagnostic, although they do not replace a population support assumption. A complementary Hessian diagnostic is also available from the explicit formulas derived in Section~\ref{scn:common-setup-derivatives}: in the rotationally symmetric case one may use the scalar formula \eqref{eqn:Hess}, while in the general-density case one may use the matrix formula \eqref{eqn:Hessian-general-density}, equivalently the bilinear form \eqref{eqn:Hessian-general-bilinear-form}. Thus, for spherical data, empirical radii and directions can be plugged into these formulas to assess proximity to the Hessian-degenerate regime in which classical Fr\'echet mean CLTs and Hessian-corrected Wald-type procedures may fail. This is only a diagnostic interpretation, but it explains why sharp support thresholds and explicit Hessian formulas are useful for inference.

Spherical data are a natural testing ground for this question.  Directional and spherical observations occur throughout directional statistics, with applications in earth sciences, astronomy, biology, medicine, neuroscience, acoustics, image analysis, machine learning and related fields; see, for example, \citet{FisherLewisEmbleton1993SphericalData,MardiaJupp2000DirectionalStatistics,PewseyGarciaPortugues2021Directional}.  At the same time, \(\mathbb S^m\) is geometrically explicit enough that the support threshold can be computed rather than merely bounded.

\subsection{Smeariness and the support question}
\begin{definition}[\textbf{Lifted Fr\'echet function}]
\label{def:lifted-Frechet-function}
Let \(\mubar\) be a local Fr\'echet mean of \(X\), and let \(\exp_{\mubar}\) be the Riemannian exponential map at \(\mubar\). On a neighborhood of \(0\in T_{\mubar}M\) where normal coordinates are used, define
\[
    \widetilde F_\mu(u):=F_\mu(\exp_{\mubar}u).
\]
We call \(\widetilde F_\mu\) the \emph{lifted Fr\'echet function} at \(\mubar\).
\end{definition}

The terminology in the next definition is \textit{local} and \textit{Hessian-based}. It is designed to isolate the Hessian-degenerate regime relevant for \(2\)-smeariness and for the support obstructions proved in this paper. In the more general asymptotic terminology of \citet{HotzHuckemann2015,HE,Eltzner2022GeometricalSmeariness}, smeariness is usually formulated through the leading order of the lifted Fr\'echet function, or equivalently through the resulting nonclassical fluctuation rate of sample Fr\'echet means under appropriate CLT assumptions. Under the smooth cut-locus-avoidance assumptions used below, the \(2\)-smeary cases relevant here are Hessian-degenerate local minima whose first nonzero growth along the degenerate directions is quartic.

\begin{definition}[{\textbf{Full and directional smeariness}}]\label{def:smeariness-directional}
Assume that \(\mubar\) is an \textit{isolated} local minimum of \(F_\mu\), that \(\widetilde F_\mu\) is smooth near \(0\), and set the Hessian \(H:=D^2\widetilde F_\mu(0)\).  We say that \(\mu\) is \emph{smeary at \(\mubar\)} if \(H\) is singular.  If \(H=0\), we say that \(\mu\) is \emph{fully smeary}.  If \(H\) is singular but \(1\le \operatorname{rank}(H)<m\), we say that \(\mu\) is \emph{directionally smeary}.  More precisely, let \(V:=\ker H\). If \(V\neq\{0\}\) and
\[
    D^3\widetilde F_\mu(0)[v,v,v]=0
    \quad\text{for all }v\in V,
\]
while
\[
    D^4\widetilde F_\mu(0)[v,v,v,v]>0
    \quad\text{for all }v\in V\setminus\{0\},
\]
then we say that \(\mu\) is \(2\)-smeary along \(V\). If \(H=0\), so that
\(V=T_{\mubar}M\), we say that \(\mu\) is fully \(2\)-smeary; if
\(1\le\operatorname{rank}(H)<m\), we say that \(\mu\) is directionally
\(2\)-smeary along \(\ker H\). 
\end{definition}

In the main results, the density is assumed to vanish in a neighborhood of the cut locus of the proposed Fr\'echet mean; see Assumption~\ref{assum:density-assumption}.  Under this oft-used assumption, \(\widetilde F_\mu\) is smooth near \(0\), so the Hessian and fourth-order derivatives used in the definition and in the support thresholds below are well defined.

In the relevant work \citet{afsari2011riemannianLp}, the author established, under suitable curvature and injectivity-radius bounds, that support in a sufficiently small strongly convex ball ensures positive definiteness of $H$ above.  On $\mathbb{S}^m$, the support established there becomes sub-hemispherical, i.e. contained in an \textit{open} hemispherical ball. Our aim is different and more specific: for absolutely continuous measures on \(\mathbb S^m\), we determine sharper sphere-specific thresholds for Hessian degeneracy, full smeariness and directional smeariness.  In particular, in the rotationally symmetric case the Hessian remains positive definite \textit{beyond} the closed hemisphere, up to the radius \(R_m>\pi/2\), and the fourth-order obstruction is governed by a larger radius \(S_m\).

Let us also indicate how the present paper differs from the closest sphere-specific predecessors. Tran (2021) computed higher derivative tensors of the Fréchet function on spheres and obtained CLT consequences as well as rotationally symmetric smeary examples. \citet{Eltzner2022GeometricalSmeariness} exhibited geometrical smeariness on spheres and high-dimensional super-hemispherical constructions. The present paper has a different emphasis: it is a support-threshold result. It identifies the geometric radii \(R_m\) and \(S_m\), proves support obstructions below them, and proves sharpness by constructing absolutely continuous densities beyond the relevant threshold. It also separates the rotationally symmetric case, where full and directional smeariness coincide, from the general-density case, where directional smeariness can occur just beyond the hemisphere.

\subsection{Main results}

The key simplification under rotational symmetry is that geometry and probability separate.  The Hessian and fourth-order terms of the lifted Fr\'echet function reduce to one-dimensional radial integrals against two geometry-only functions \(b_m\) and \(h_m\), defined in \eqref{eqn:b} and \eqref{eqn:h}.  The zero \(R_m\) of \(b_m\) is the Hessian threshold; see Proposition~\ref{prop:zero_bm}.  For \(m\ge4\), the first zero \(S_m\) of \(h_m\) is the quartic threshold; see Proposition~\ref{prop:sign_h_m_2_3_and_zero_m_ge_4}.  The ordering \(R_m<S_m\) and the sign of \(h_m\) below \(S_m\) are proved in Lemma~\ref{lem:Rm-Sm-ordering}, and the asymptotic relation \(R_m,S_m=\pi/2+O(1/m)\) is recorded in Proposition~\ref{prop:asymptotics-Rm-Sm}.

The first main theorem, Theorem~\ref{thm:semi-tight-support-rot-symm}, gives the rotationally symmetric support characterization.  In dimensions \(m=2,3\), rotationally symmetric smeariness is impossible under the standing assumptions.  For \(m\ge4\), support contained in the closed ball of radius \(S_m\) around the Fr\'echet mean still cannot produce a smeary local minimum.  Conversely, for every \(\varepsilon>0\) with \(S_m+\varepsilon<\pi\), there exists an absolutely continuous rotationally symmetric measure supported in \(B_g(\mubar;S_m+\varepsilon)\), with positive mass beyond \(S_m\), whose Hessian vanishes and whose fourth-order term is positive in every tangent direction.  Thus \(S_m\) is the \textit{sharp} support threshold for full rotationally symmetric \(2\)-smeariness.

The second main theorem, Theorem~\ref{thm:general-density-support-threshold}, treats general absolutely continuous densities.  Without rotational symmetry, full and directional smeariness are no longer equivalent.  Closed hemispherical support forces the Hessian to be positive definite and hence rules out both full and directional smeariness.  Support contained in the closed ball of radius \(R_m\) prevents the Hessian from vanishing identically, and support contained in the closed ball of radius \(S_m\) rules out full smeariness.  These obstructions are sharp in different senses: directional \(2\)-smeariness can occur with support in \(B_g(N;\pi/2+\varepsilon)\) for arbitrary \(\varepsilon>0\), while full \(2\)-smeariness can occur with support in \(B_g(N;S_m+\varepsilon)\), even within the rotationally symmetric subclass.
\subsection{Organization of the paper} The rest of the paper is organized as follows.  Section~\ref{scn:TE-sqdist} derives the fourth-order Taylor expansion of the squared distance on \(\mathbb S^m\).  Section~\ref{scn:common-setup-derivatives} sets up the general-density and rotationally symmetric Fr\'echet functions and derives the Hessian and fourth-order formulas.  Section~\ref{scn:asymptotics-b-h} studies the radial threshold functions \(b_m,h_m\) and their zeros \(R_m,S_m\).  Section~\ref{scn:semi-tightness-rot-symm} proves the rotationally symmetric support theorem, and Section~\ref{scn:general-density-support-threshold} proves the general-density support theorem. Section~\ref{scn:bpclt-consequences} records BPCLT consequences in the support regimes where the Hessian is forced to be positive definite: below \(R_m\) in the rotationally symmetric case and below \(\pi/2\) in the general-density case. Section~\ref{app:asymptotics-Rm-Sm} records the high-dimensional location of the thresholds, and Appendix~\ref{scn:symbolic-computation} contains the symbolic computations supporting the Taylor expansion.

\subsection{{Acknowledgements and funding.}}

This research was supported by funding from Research Foundation--Flanders (FWO) via the Odysseus II programme no.~G0DBZ23N. The author acknowledges technical discussions with Prof. Dr. Stephan Huckemann and Dr. Benjamin Eltzner from the University of G\"ottingen, Germany, and also encouragement by Junior Professor David Tewodrose at Vrije Universiteit Brussel.


\section{Taylor expansion of the Riemannian squared distance on $\mathbb{S}^m$ through order $4$}
\label{scn:TE-sqdist}

This section records the Taylor expansion, in normal coordinates at the north pole, of the squared Riemannian distance on the sphere.  The expansion is the local computational input used later to derive the Hessian and fourth-order terms of the Fr\'echet function.  Throughout, the expansion is taken in the first variable while the second variable is fixed away from the cut locus.

\subsection{Taylor expansion of the Riemannian squared distance function in the tangent space using the exponential map}

Let $\mathbb{S}^{m}\subset\mathbb{R}^{m+1}$ be the unit sphere, let $N\in\mathbb{S}^{m}$ be the north pole, and identify $T_{N}\mathbb{S}^{m}\cong\mathbb{R}^{m}$ with the Euclidean inner product $\langle\cdot,\cdot\rangle$ and norm $\|\cdot\|$.  Fix $v\in T_{N}\mathbb{S}^{m}$ with \(R:=\|v\|\in(0,\pi)\).  Define the pulled-back Riemannian squared distance function on \(T_N\mathbb S^m\) by
\begin{equation}\label{eqn:pulled-back-sq-dist}
   g(u;v):=d^{2}\big(\exp_{N}u,\exp_{N}v\big). 
\end{equation}

For $a\in T_{N}\mathbb{S}^{m}$ with $\|a\|=1$, set
\[
\alpha:=\langle a,v\rangle,
\qquad
f_{a}(t):=g(ta;v)=d^{2}\big(\exp_{N}(ta),\exp_{N}v\big).
\]
Thus \(f_a\) is the line restriction of \(g(\cdot;v)\) along \(a\).  By the spherical cosine law in normal coordinates,
\begin{equation}\label{eqn:f_a}
\cos\big(d(\exp_{N}(ta),\exp_{N}v)\big)
=
\cos t\,\cos R
+
\sin t\,\frac{\sin R}{R}\,\alpha.
\end{equation}
Writing \(d(t):=d(\exp_{N}(ta),\exp_{N}v)\in(0,\pi)\), we have
\[
f_{a}(t)=d(t)^{2},
\qquad
d(t)=\arccos\!\Big(\cos t\,\cos R+\sin t\,\frac{\sin R}{R}\,\alpha\Big).
\]
A Taylor expansion of \(f_a\) at \(t=0\) gives
\[
f_{a}(t)=\sum_{k=0}^{4} C_{k}(R,\alpha)\,t^{k}+O(t^{5}),
\qquad (t\to 0),
\]
where
\begin{equation}\label{eqn:C}
\boxed{
\begin{aligned}
C_{0}(R,\alpha)&=R^{2},
\qquad
C_{1}(R,\alpha)=-2\alpha,\\
C_{2}(R,\alpha)&=\frac{R}{\tan R}+\alpha^{2}\Big(\frac{1}{R^{2}}-\frac{1}{R\tan R}\Big),\\
C_{3}(R,\alpha)&=A_{3}(R)\,\alpha+B_{3}(R)\,\alpha^{3},\\
C_{4}(R,\alpha)&=A_{4}(R)\,\alpha^{4}+A_{2}(R)\,\alpha^{2}+A_{0}(R).
\end{aligned}
}
\end{equation}
Here
\begin{equation}\label{eqn:A-B}
\boxed{
\begin{aligned}
A_{3}(R)&= - \frac{2}{3} + \frac{1}{\sin^{2} R} - \frac{1}{R \tan^{5} R}
- \frac{2\cos R}{R \sin^{3} R} + \frac{\cos R}{R \sin^{5} R},
\qquad
B_{3}(R)= \frac{2}{3 R^{2}} - \frac{1}{R^{2}\sin^{2} R} + \frac{1}{R^{3}\tan R},\\
A_{4}(R)&=\frac{\frac{6 R}{\tan R} - \frac{15 R \cos R}{\sin^{3} R} - 11 + \frac{15}{\sin^{2} R}}{12 R^{4}},\\
A_{2}(R)&=\frac{- \frac{4 R}{\tan^{5} R} + \frac{R \cos R}{\sin^{3} R} + \frac{4 R \cos R}{\sin^{5} R}
- \frac{7}{\tan^{6} R} + \frac{12}{\sin^{2} R} - \frac{21}{\sin^{4} R} + \frac{7}{\sin^{6} R}}{6 R^{2}},\\
A_{0}(R)&=
- \frac{R}{6 \tan^{7} R} + \frac{R \cos^{5} R}{4 \sin^{7} R} - \frac{R \cos R}{12 \sin^{7} R}
- \frac{1}{4 \tan^{8} R} - \frac{3}{4 \sin^{4} R} + \frac{5}{4 \sin^{6} R}
+ \frac{3 \cos^{6} R}{4 \sin^{8} R} - \frac{1}{2 \sin^{8} R}.
\end{aligned}
}
\end{equation}
The symbolic calculation of these Taylor coefficients is recorded in Section~\ref{scn:symbolic-computation}.

For each \(k\ge 1\) such that \(g(\cdot;v)\) is \(k\) times differentiable at \(0\), the line-restriction identity gives
\[
f_{a}^{(k)}(0)=D^{k}g(0;v)[a,\dots,a]=k!\,C_{k}(R,\alpha).
\]

Now write \(u\neq 0\) as \(u=ra\), with \(r:=\|u\|\) and \(\|a\|=1\).  Then \(g(ra;v)=f_a(r)\), and \(\alpha=\langle a,v\rangle=\langle u,v\rangle/\|u\|\).  Hence the expansion above yields
\[
g(u;v)=\sum_{k=0}^{4} P_{k}(u;v)+O(\|u\|^{5}),
\qquad
P_{k}(u;v):=\frac{1}{k!}D^{k}g(0;v)[u,\dots,u],
\]
where the homogeneous terms \(P_k\) are given by
\[
P_{k}(u;v)
= \|u\|^{k} \, C_k(R, \alpha)
=\|u\|^{k}\,C_{k}\!\left(R,\Big\langle\frac{u}{\|u\|},\,v\Big\rangle\right),
\qquad R:=\|v\|.
\]
In particular, for fixed \(v\) the expansion through degree \(4\) is
\begin{equation}\label{eqn:g-eqn}
\boxed{
\begin{aligned}
g(u;v)
&=R^{2}
-2\langle u,v\rangle\\[1mm]
&\quad+\Bigg(\frac{R}{\tan R}\,\|u\|^{2}
+\Big(\frac{1}{R^{2}}-\frac{1}{R\tan R}\Big)\,\langle u,v\rangle^{2}\Bigg)\\[1mm]
&\quad+\Big(A_{3}(R)\,\langle u,v\rangle\,\|u\|^{2}+B_{3}(R)\,\langle u,v\rangle^{3}\Big)\\[1mm]
&\quad+\Big(A_{4}(R)\,\langle u,v\rangle^{4}+A_{2}(R)\,\langle u,v\rangle^{2}\,\|u\|^{2}+A_{0}(R)\,\|u\|^{4}\Big)
+O(\|u\|^{5}),
\qquad R:=\|v\|.
\end{aligned}
}
\end{equation}
Here \(A_{3},B_{3},A_{4},A_{2},A_{0}\) are the functions of \(R\) defined in \eqref{eqn:A-B}.

\subsection{Second and fourth partial derivative tensors of the Riemannian squared distance in normal coordinates}
\label{subsec:derivatives-g}

We next extract from \eqref{eqn:g-eqn} the second and fourth derivatives of \(g(\cdot;v)\) at \(0\).  Fix \(v=R\Theta\), where \(R:=\|v\|\in(0,\pi)\) and \(\Theta\in\mathbb{S}^{m-1}\).  The derivatives below are always taken with respect to the \(u\)-variable, with \(v\) fixed.

Since the fourth derivative is a symmetric \(4\)-linear form, it is enough to know its values on the diagonal.  We record the elementary polarization fact used for this purpose.

\begin{definition}[\textbf{Symmetric $k$--linear form}]
Let \(V\) be a real vector space and let \(k\in\mathbb{N}\).  A \(k\)-linear form \(B:V^k\to\mathbb{R}\) is called symmetric if, for every permutation \(\sigma\in S_k\) and all \(w_1,\dots,w_k\in V\),
\[
B(w_1,\dots,w_k)=B(w_{\sigma(1)},\dots,w_{\sigma(k)}).
\]
\end{definition}

\begin{lemma}[\textbf{Polarization identity}]
\label{lem:polarization}
Let \(k\in\mathbb{N}\), and let \(B:V^k\to\mathbb{R}\) be a symmetric \(k\)-linear form.  Define the diagonal polynomial \(p:V\to\mathbb{R}\) by
\[
p(w):=B(w,\dots,w).
\]
Then \(B\) is uniquely determined by \(p\).  In particular, if \(p\equiv0\), then \(B\equiv0\).  Moreover,
\[
B(w_1,\dots,w_k)
=\frac{1}{2^k k!}\sum_{\varepsilon_i=\pm 1}
\Bigl(\prod_{i=1}^k \varepsilon_i\Bigr)\,
p\!\Bigl(\sum_{i=1}^k \varepsilon_i w_i\Bigr).
\]
\end{lemma}
\begin{proof}
    Standard.
\end{proof}

\subsection{Explicit formulas for the Hessian and the fourth-order derivative of the squared distance}
In the Euclidean vector space \(T_NS^m\), one has
\(D^2_u\|u\|^2[w,w]=2\|w\|^2\) and
\(D^2_u\langle u,v\rangle^2[w,w]=2\langle w,v\rangle^2
=2R^2\langle w,\Theta\rangle^2\), where \(v=R\Theta\).
The quadratic and quartic parts of \eqref{eqn:g-eqn} therefore give, for every \(w\in T_N\mathbb S^m\),
\begin{equation}\label{eqn:Hess-g}
\boxed{
 D_u^2 g(0;R\Theta)[w,w]
=
2\frac{R}{\tan R}\|w\|^2
+
2\Bigl(1-\frac{R}{\tan R}\Bigr)\langle w,\Theta\rangle^2, 
}
\end{equation}

After writing \(v=R\Theta\), the homogeneous quartic part of the expansion above is
\[
    Q_4(u;R\Theta)
    =
    A_0(R)\|u\|^4
    +
    A_2(R)R^2\|u\|^2\langle u,\Theta\rangle^2
    +
    A_4(R)R^4\langle u,\Theta\rangle^4 .
\]
By the Taylor convention,
\[
    Q_4(z;R\Theta)
    =
    \frac{1}{4!}D_u^4g(0;R\Theta)[z,z,z,z],
    \qquad z\in T_N\mathbb S^m .
\]
Taking \(z=w\) gives

\begin{equation}\label{eqn:fourth-derivative-g}
\boxed{
   D_u^4 g(0;R\Theta)[w,w,w,w]
=
24\Bigl(
A_0(R)\|w\|^4
+
A_2(R)R^2\|w\|^2\langle w,\Theta\rangle^2
+
A_4(R)R^4\langle w,\Theta\rangle^4
\Bigr). 
}
\end{equation}

By Lemma~\ref{lem:polarization}, the full symmetric \(4\)-linear tensor \(D_u^4 g(0;R\Theta)\) is \textit{uniquely} determined by the displayed diagonal formula.


\section{Setup and derivative formulas}
\label{scn:common-setup-derivatives}
\label{scn:support-rot-symm}
\label{scn:support-general-density}

This section derives \textit{explicit formulas} for the Hessian (cf. Equation \eqref{eqn:Hessian-general-density}) and the fourth-order derivative tensors (cf. Equation \eqref{eqn:fourth-derivative-general-density}) for the population Fréchet function, assuming that the density assigns zero probability mass to a neighborhood of the cut locus. We first treat an arbitrary absolutely continuous density on \(\mathbb S^m\), not assumed to be rotationally symmetric, and then obtain the rotationally symmetric formulas as reductions. The pointwise expansion of the Riemannian squared-distance \(g(u;v)=d^2(\exp_Nu,\exp_Nv)\) is the one derived in Section~\ref{scn:TE-sqdist}; in particular, we use the coefficients \(A_0,A_2,A_4\) from \eqref{eqn:A-B} and the identities \eqref{eqn:Hess-g} and \eqref{eqn:fourth-derivative-g}.

\subsection{General densities and rotational symmetry}
\label{scn:set-up}
\label{scn:set-up-general-density}

Let \(\mu\) be an absolutely continuous probability measure on \(\mathbb S^m\).  After applying an isometry, we take the population Fréchet mean to be the North Pole \(N\).  Since the cut locus \(\operatorname{Cut}(N)=\{-N\}\) and \(\mu\) is absolutely continuous, \(\mu(\{-N\})=0\); also see Assumption \ref{assum:density-assumption}.  Thus \(V:=\exp_N^{-1}(X)\) is defined \(\mu\)-a.s., and we write \(V=R\Theta\), with \(R\in(0,\pi)\) and \(\Theta\in\mathbb S^{m-1}\).  In these coordinates, \(d\operatorname{vol}_{\mathbb S^m}=(\sin R)^{m-1}\,dR\,d\Theta\).

Assume that \(\widetilde\mu:=(\exp_N^{-1})_\#\mu\) has the following density (i.e. Radon-Nikodym derivative w.r.t. $d\operatorname{vol}_{\mathbb S^m}$) in the polar coordinates:
\begin{equation}
\label{eqn:general-density-polar}
    d\widetilde\mu(R,\Theta)=\rho(R,\Theta)(\sin R)^{m-1}\,dR\,d\Theta,
    \qquad
    \int_0^\pi\!\!\int_{\mathbb S^{m-1}}\rho(R,\Theta)(\sin R)^{m-1}\,d\Theta\,dR=1 .
\end{equation}
The \textit{lifted} Fréchet function in these polar coordinates is
\begin{equation}\label{eqn:Frechet-general-polar}
    \widetilde F_\mu(u):=F_\mu(\exp_Nu)
    =
    \int_0^\pi\!\!\int_{\mathbb S^{m-1}}
    g(u;R\Theta)\rho(R,\Theta)(\sin R)^{m-1}\,d\Theta\,dR .
\end{equation}
The rotationally symmetric case is the specialization \(\rho(R,\Theta)=\rho(R)\), giving
\begin{equation}
\label{eqn:Fréchet-rot-symm-double-int}
    \widetilde F_\mu(u)
    =
    \int_0^\pi\!\!\int_{\mathbb S^{m-1}}
    g(u;R\Theta)\rho(R)(\sin R)^{m-1}\,d\Theta\,dR .
\end{equation}

\begin{assumption}[\textbf{Cut-locus avoidance}]
\label{assum:density-assumption}
We assume that there exists \(\varepsilon_0>0\) such that \(\rho(R,\Theta)=0\) for a.e. \((R,\Theta)\) with \(R\in(\pi-\varepsilon_0,\pi)\).  In the rotationally symmetric case this means \(\rho(R)=0\) for \(R\) sufficiently close to \(\pi\).  Under this assumption, \(\widetilde F_\mu\) is smooth near \(0\in T_N\mathbb S^m\), and differentiation under the integral sign at \(u=0\) is justified; see Proposition~3.2.1 of \citet{Tran2020Sampling}. There the author justified this for order up to $4,$ but the same argument carries over to any order of derivatives.
\end{assumption}

\subsection{Formula for the Hessian of the Fréchet function, and its rotationally symmetric reduction}
\label{para:Hess-general-density}

\begin{proposition}[\textbf{Formula for the Hessian of the Fréchet function}]
\label{prop:Hessian-common}
Assume Assumption~\ref{assum:density-assumption}, and set the Hessian of the Fréchet function $F$ at $\mubar$, or equivalently, the Hessian of the lifted Fréchet function $\widetilde{F}_\mu$ at $0\in T_{\mubar}\mathbb{S}^m$ to be  \(H:=D^2\widetilde F_\mu(0)\).  Then
\begin{equation}
\label{eqn:Hessian-general-density}
\boxed{
\begin{aligned}
    H
    =
    2\int_0^\pi\!\!\int_{\mathbb S^{m-1}}
    \left[
    \frac{R}{\tan R}I_m+
    \left(1-\frac{R}{\tan R}\right)\Theta\Theta^{\mathsf T}
    \right]
    \rho(R,\Theta)(\sin R)^{m-1}\,d\Theta\,dR .
\end{aligned}}
\end{equation}
Equivalently,
\begin{equation}
\label{eqn:Hessian-general-bilinear-form}
    w^{\mathsf T}Hw
    =
    2\int_0^\pi\!\!\int_{\mathbb S^{m-1}}
    \left[
    \frac{R}{\tan R}\|w\|^2+
    \left(1-\frac{R}{\tan R}\right)\langle w,\Theta\rangle^2
    \right]    \rho(R,\Theta)(\sin R)^{m-1}\,d\Theta\,dR .
\end{equation}
In the rotationally symmetric case, define
\begin{equation}\label{eqn:b}
    b_m(R):=1+(m-1)R\cot R .
\end{equation}
Then
\begin{equation}
\label{eqn:Hess}
\boxed{
    H
    =
    \frac{2\operatorname{vol}(\mathbb S^{m-1})}{m}
    \left(\int_0^\pi \rho(R)(\sin R)^{m-1}b_m(R)\,dR\right)I_m .
}
\end{equation}
Thus, under rotational symmetry, the Hessian is singular\textit{ if and only if} it vanishes \textit{if and only if }the scalar integral in \eqref{eqn:Hess} vanishes.
\end{proposition}

\begin{proof}
By \eqref{eqn:Frechet-general-polar}, Assumption~\ref{assum:density-assumption}, and the pointwise Hessian formula \eqref{eqn:Hess-g}, differentiating under the integral sign gives \eqref{eqn:Hessian-general-bilinear-form}, hence \eqref{eqn:Hessian-general-density}.  If \(\rho(R,\Theta)=\rho(R)\), then using the Funk--Hecke formula (see \cite{ChafaiFunkHecke} for example)
\[
    \int_{\mathbb S^{m-1}}\langle w,\Theta\rangle^2\,d\Theta
    =
    \frac{\operatorname{vol}(\mathbb S^{m-1})}{m}\|w\|^2,
\]
 \eqref{eqn:Hess} follows.  The final assertion follows because the Hessian is then a scalar multiple of \(I_m\).
\end{proof}

\subsection{Fourth-order derivatives of the Fréchet function and rotationally symmetric reduction}
\label{para:fourth-derivative-general-density}

\begin{proposition}[\textbf{Fourth derivative formula}]
\label{prop:fourth-common}
\label{prop:Hess-fourth-Fréchet-rot-symm}
\label{prop:Hess-fourth-Frechet-rot-symm}
Assume Assumption~\ref{assum:density-assumption}.  Then, for every \(w\in T_N\mathbb S^m\), the fourth derivative of the lifted Fréchet function is given by:
\begin{equation}
\label{eqn:fourth-derivative-general-density}
\boxed{
\begin{aligned}
    D^4\widetilde F_\mu(0)[w,w,w,w]
    =
    24\int_0^\pi\!\!\int_{\mathbb S^{m-1}}
    \Big[
    &A_0(R)\|w\|^4+
    A_2(R)R^2\|w\|^2\langle w,\Theta\rangle^2 \\
    &+
    A_4(R)R^4\langle w,\Theta\rangle^4
    \Big]
    \rho(R,\Theta)(\sin R)^{m-1}\,d\Theta\,dR .
\end{aligned}}
\end{equation}
Let \(d\sigma(w)\) be \textit{normalized} surface measure on \(\mathbb S^{m-1}\subset T_N\mathbb S^m\), and define
\begin{equation}
\label{eqn:h}
    h_m(R):=
    A_0(R)+\frac{1}{m}A_2(R)R^2+\frac{3}{m(m+2)}A_4(R)R^4 .
\end{equation}
Then the averaged fourth derivative on $\mathbb S^{m-1}$ is given by:
\begin{equation}
\label{eqn:averaged-fourth-derivative-general-density}
\boxed{
    \int_{\mathbb S^{m-1}}D^4\widetilde F_\mu(0)[w,w,w,w]\,d\sigma(w)
    =
    24\int_0^\pi\!\!\int_{\mathbb S^{m-1}}
    h_m(R)\rho(R,\Theta)(\sin R)^{m-1}\,d\Theta\,dR .
}
\end{equation}
In the rotationally symmetric case,
\begin{equation}
\label{eqn:fourth-derivative}
\boxed{
    D^4\widetilde F_\mu(0)[w,w,w,w]
    =
    24\operatorname{vol}(\mathbb S^{m-1})
    \left(\int_0^\pi \rho(R)(\sin R)^{m-1}h_m(R)\,dR\right)\|w\|^4 .
}
\end{equation}
Thus, under rotational symmetry, the fourth derivative tensor vanishes if and only if the scalar integral in \eqref{eqn:fourth-derivative} vanishes.
\end{proposition}

\begin{proof}
Equation \eqref{eqn:fourth-derivative-general-density} follows by integrating the pointwise fourth derivative formula \eqref{eqn:fourth-derivative-g} against \eqref{eqn:general-density-polar}.  For fixed \(\Theta\), the normalized spherical moments are again immediate consequences of the Funk-Hecke formula, see \cite{ChafaiFunkHecke} for example:
\[
    \int_{\mathbb S^{m-1}}\langle w,\Theta\rangle^2\,d\sigma(w)=\frac1m,
    \qquad
    \int_{\mathbb S^{m-1}}\langle w,\Theta\rangle^4\,d\sigma(w)=\frac{3}{m(m+2)}.
\]
This gives \eqref{eqn:averaged-fourth-derivative-general-density}.  If \(\rho(R,\Theta)=\rho(R)\), the corresponding unnormalized identities give \eqref{eqn:fourth-derivative}.  The final assertion follows from Lemma~\ref{lem:polarization}, since the diagonal quartic form is a scalar multiple of \(\|w\|^4\).
\end{proof}

\section{Radial threshold functions \(b_m,h_m\) and their zeros}
\label{scn:asymptotics-b-h}

Under rotational symmetry, \eqref{eqn:Hess} and \eqref{eqn:fourth-derivative} show that the Hessian and fourth derivative of the lifted Fréchet function are governed by the radial functions \(b_m\) and \(h_m\), defined in \eqref{eqn:b} and \eqref{eqn:h}. We record only the sign properties and zeros \(R_m,S_m\) of these functions that are needed for the fixed-dimensional support theorems. The high-dimensional location of \(R_m\) and \(S_m\) is not used in the main argument and is recorded in Section~\ref{app:asymptotics-Rm-Sm}.
\subsection{\textbf{Analysis of $b_m$.}}
\begin{proposition}[\textbf{The Hessian threshold}]
\label{prop:zero_bm}
Let \(m\ge2\), and following \eqref{eqn:b}, let \(b_m:(0,\pi)\to\mathbb R\) be defined by
\[
    b_m(R)=1+(m-1)R\cot R.
\]
Then \(b_m\) is strictly decreasing on \((0,\pi)\), has a unique zero
\begin{equation}\label{eqn:R_m}
    \boxed{R_m\in(\pi/2,\pi),}
\end{equation}

and satisfies \(b_m(R)>0\) for \(R\in(0,R_m)\) and \(b_m(R)<0\) for \(R\in(R_m,\pi)\).
\end{proposition}

\begin{proof}
Set \(q(R):=R\cot R\). Then
\[
    q'(R)=\cot R-R\csc^2R=\frac{\sin R\cos R-R}{\sin^2R}.
\]
The numerator \(s(R):=\sin R\cos R-R\) satisfies \(s(0)=0\) and
\(s'(R)=\cos(2R)-1<0\) on \((0,\pi)\). Hence \(q'(R)<0\), so \(b_m\) is strictly decreasing.
Moreover, \(q(R)\to1\) as \(R\downarrow0\), while \(q(R)\to-\infty\) as \(R\uparrow\pi\).
Thus \(b_m\) has a unique zero. Since \(q(\pi/2)=0\), this zero lies in \((\pi/2,\pi)\), and the
stated signs follow from strict monotonicity.
\end{proof}
\subsection{\textbf{Analysis of $h_m$}}

For the quartic coefficient \(h_m\), substituting the explicit coefficient functions \(A_0,A_2,A_4\)
from \eqref{eqn:A-B} into \eqref{eqn:h} gives the following simplified form. With \(c=\cot R\),
\begin{equation}
\label{eqn:h_m_simplified}
\boxed{
    h_m(R)=
    -\frac{m-1}{12m(m+2)}
    \left\{
    3(m-3)Rc^3+(m-7)Rc-3(m-3)c^2+4
    \right\}.
}
\end{equation}
In particular \(h_m(\pi/2)=-(m-1)/(3m(m+2))<0\), by continuous extension.

\begin{proposition}[\textbf{Sign and zeros of the radial quartic coefficient \(h_m\)}]
\label{prop:sign_h_m_2_3_and_zero_m_ge_4}
The following hold.

\emph{(i)} \(h_2(R)<0\) for all \(R\in(0,\pi)\).

\emph{(ii)} \(h_3(R)<0\) for all \(R\in(0,\pi)\).

\emph{(iii)} If \(m\ge4\), then \(h_m\) has a zero in \((\pi/2,\pi)\). We denote the first zero by
\begin{equation}\label{eqn:S_m}
     \boxed{
    S_m:=\inf\{R\in(\pi/2,\pi):h_m(R)=0\}.
    }
\end{equation}

\end{proposition}

\begin{proof}
(i) For \(m=2\), direct substitution gives
\[
    h_2(R)=
    \frac{R\cos R(3+2\sin^2R)-\sin R(3+\sin^2R)}{96\sin^3R}.
\]
The denominator is positive on \((0,\pi)\). On \((\pi/2,\pi)\), the numerator is plainly negative.
On \((0,\pi/2)\), set \(F(R):=(3+\sin^2R)\tan R-R(3+2\sin^2R)\). Then the numerator is
\(-\cos R\,F(R)\), and
\[
    F'(R)=4\tan R\left(\tan R-\frac{R}{1+\tan^2R}\right)>0,
\]
because \(\tan R>R>R/(1+\tan^2R)\). Since \(F(0)=0\), this proves \(h_2<0\).

(ii) For \(m=3\), one obtains the exact identity \(h_3(R)=\frac{2}{45}(R\cot R-1)\). This is negative
on \((0,\pi/2)\) because \(\tan R>R\), and it is negative on \((\pi/2,\pi)\) because \(\cot R<0\).

(iii) Let \(m\ge4\). We already observed from \eqref{eqn:h_m_simplified} that \(h_m(\pi/2)<0\). Also,
as \(t=\pi-R\downarrow0\), one has \(\cot(\pi-t)=-t^{-1}+O(t)\), and therefore
\[
    h_m(\pi-t)
    =
    \frac{\pi(m-1)(m-3)}{4m(m+2)}\,t^{-3}+O(t^{-2}).
\]
The leading coefficient is positive for \(m\ge4\), so \(h_m(R)\to+\infty\) as \(R\uparrow\pi\).
By continuity, \(h_m\) has a zero in \((\pi/2,\pi)\). Since \(h_m<0\) near \(\pi/2\), its first zero
\(S_m\) is well defined and belongs to \((\pi/2,\pi)\).
\end{proof}

\begin{lemma}[\textbf{Sign of \(h_m\) below \(S_m\) and ordering of \(R_m,S_m\)}]
\label{lem:Rm-Sm-ordering}
For every \(m\ge4\),
\[
    \boxed{
    h_m(R)<0\quad\text{for all }0<R<S_m,
    \qquad
    R_m<S_m.
    }
\]
\end{lemma}

\begin{proof}
We first show that \(h_m<0\) on \((0,\pi/2)\). Put \(t:=\tan R>0\). In
\eqref{eqn:h_m_simplified}, the expression in braces, multiplied by the positive factor \(t^3\), becomes
\[
    G_m(R):=4t^3+(m-7)Rt^2-3(m-3)t+3(m-3)R.
\]
Write \(G_m=G_4+(m-4)A\), where
\(A(R):=Rt^2-3t+3R\) and \(G_4(R):=4t^3-3Rt^2-3t+3R\). A direct calculation gives
\[
    A'(R)=2t\{R(1+t^2)-t\},\qquad
    G_4'(R)=6t\{t+2t^3-R(1+t^2)\}.
\]
The functions in braces extend continuously at \(0\) and vanish there, and their derivatives are respectively
\(2Rt(1+t^2)>0\) and \(2t(1+t^2)(3t-R)>0\), using \(\tan R>R\). Hence \(A>0\) and
\(G_4>0\) on \((0,\pi/2)\), so \(G_m>0\). By \eqref{eqn:h_m_simplified}, \(h_m<0\) on
\((0,\pi/2)\).

Next consider \(R\in(\pi/2,R_m]\), and set \(y:=-R\cot R\). Proposition~\ref{prop:zero_bm}
gives \(0<y\le1/(m-1)\). The expression in braces in \eqref{eqn:h_m_simplified} equals
\[
    4-(m-7)y-\frac{3(m-3)}{R^2}(y^2+y^3).
\]
Using \(R>\pi/2\), \(y\le1/(m-1)\), and \(\pi^2>9\), this is bounded below by
\[
    4-\frac{\max\{m-7,0\}}{m-1}
      -\frac{4m(m-3)}{3(m-1)^3}>0.
\]
Indeed, for \(m<7\) the middle term is absent and the last term is \(<1\), while for \(m\ge7\) the
middle term is \(<1\) and the last term is \(<1/3\). Thus \(h_m<0\) on \((\pi/2,R_m]\). Since
\(S_m\) is the first zero of \(h_m\) in \((\pi/2,\pi)\), it follows that \(R_m<S_m\). Combining this
with the negativity on \((0,\pi/2)\) and the definition of \(S_m\) gives \(h_m(R)<0\) for all
\(0<R<S_m\).
\end{proof}

\begin{lemma}[\textbf{Right-neighborhood positivity of \(h_m\) past its first zero}]
\label{lem:positivity-h_m-past-first-zero}
Let \(m\ge4\), and let \(S_m\in(\pi/2,\pi)\) be the first zero of \(h_m\), whose existence was
established in Proposition~\ref{prop:sign_h_m_2_3_and_zero_m_ge_4}. Then
\[
    \boxed{h_m'(S_m)>0.}
\]
Consequently, there exists \(\eta_m \in(0,\pi-S_m)\) such that
\[
    h_m(R)>0\qquad\text{for all }R\in(S_m,S_m+\eta_m).
\]
\end{lemma}

\begin{proof}
For \(R\in(\pi/2,\pi)\), set \(x:=-\cot R>0\), \(a:=m-3>0\), and \(b:=m-7\). By
\eqref{eqn:h_m_simplified},
\[
    h_m(R)=\frac{m-1}{12m(m+2)}Q_m(R),
    \qquad
    Q_m(R):=3aRx^3+bRx+3ax^2-4.
\]
Since \(x'=1+x^2\),
\[
    Q_m'(R)=x(3ax^2+b)+(1+x^2)\{6ax+R(9ax^2+b)\}.
\]
If \(m\ge7\), then \(b\ge0\), so \(Q_m'(R)>0\) on \((\pi/2,\pi)\). Hence
\(Q_m'(S_m)>0 \implies h_m'(S_m) > 0\).

It remains to treat \(m=4,5,6\), where \(b<0\). Let \(x_m:=-\cot S_m\). Since \(Q_m(S_m)=0\),
we claim that \(3ax_m^2+b>0\). Indeed, if \(3ax_m^2+b\le0\), then
\[
    0=Q_m(S_m)
    =S_mx_m(3ax_m^2+b)+3ax_m^2-4
    \le 3ax_m^2-4
    \le -b-4=3-m<0,
\]
a contradiction. Thus \(3ax_m^2+b>0\). Since \(b<0\), we also have
\(9ax_m^2+b=3(3ax_m^2+b)-2b>0\). Substituting \(R=S_m\) and \(x=x_m\) into the formula for
\(Q_m'\) gives \(Q_m'(S_m)>0\). Therefore \(h_m'(S_m)>0\) for every \(m\ge4\), and the final
right-neighborhood positivity follows by continuity.
\end{proof}


\section{Sharp support thresholds for smeariness under rotational symmetry}
\label{scn:semi-tightness-rot-symm}
\label{scn:hemisphere-obstruction-rot-symm}
\label{scn:construct-smeary-rot-symm}

We now combine the rotationally symmetric derivative formulas with the radial threshold properties of \(b_m,h_m,R_m,S_m\). The Hessian is controlled by \(b_m\) through \eqref{eqn:Hess}; the fourth-order term is controlled by \(h_m\) through \eqref{eqn:fourth-derivative}. Proposition~\ref{prop:zero_bm} gives the unique Hessian threshold \(R_m\in(\pi/2,\pi)\). For \(m\ge4\), Proposition~\ref{prop:sign_h_m_2_3_and_zero_m_ge_4} gives the first zero \(S_m\in(\pi/2,\pi)\) of \(h_m\). Moreover, Lemma~\ref{lem:Rm-Sm-ordering} gives \(R_m<S_m\) and \(h_m<0\) on \((0,S_m)\), while Lemma~\ref{lem:positivity-h_m-past-first-zero} gives \(h_m'(S_m)>0\), hence positivity of \(h_m\) immediately to the right of \(S_m\).
\begin{figure}[htbp]
    \centering
    \includegraphics[width=0.5\textwidth]{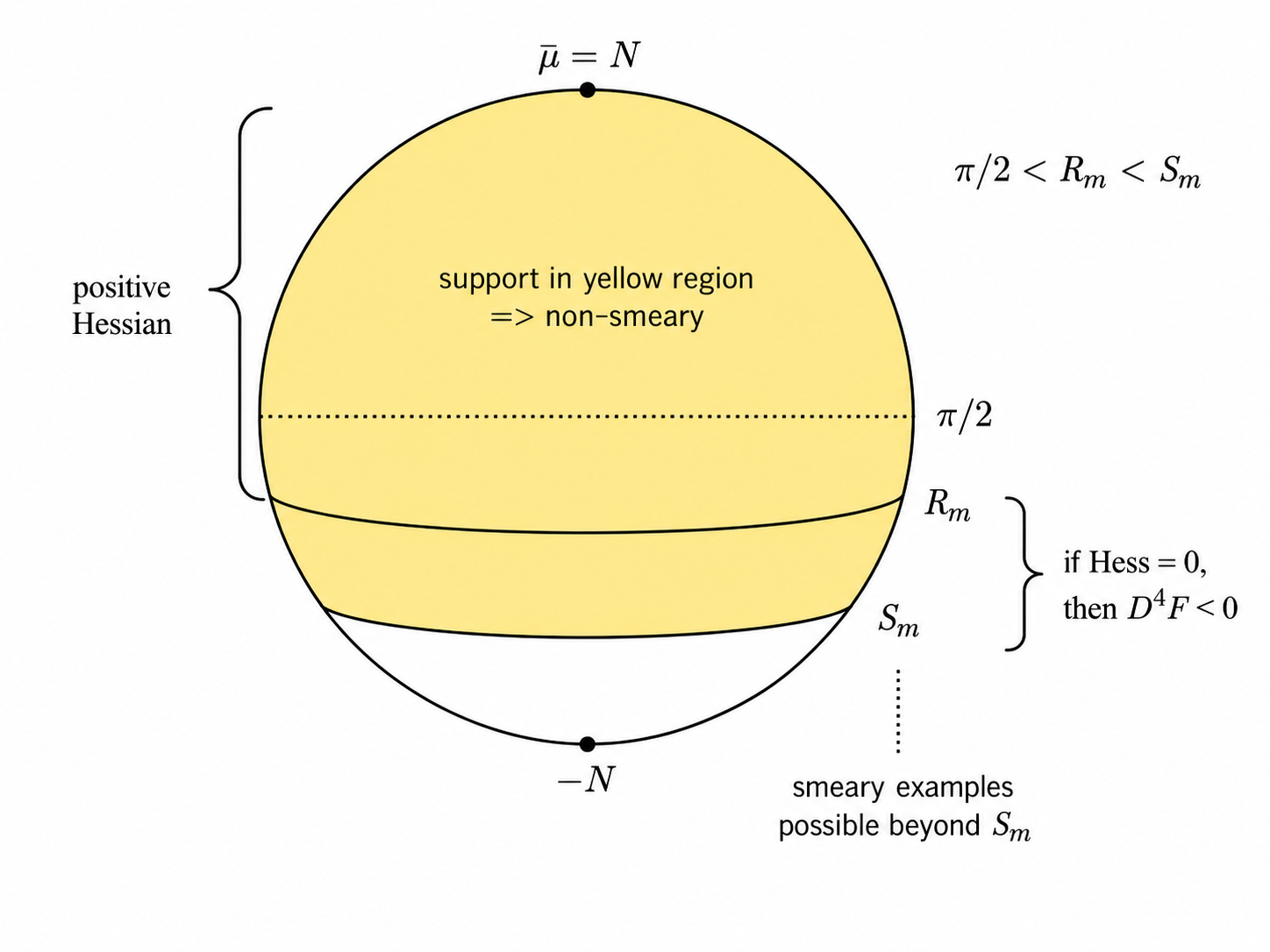}
    \caption{Support-dependent behavior under rotational symmetry. If the radial support remains below the quartic threshold \(S_m\), rotationally symmetric smeariness at \(N\) is ruled out; for every \(m\ge4\), support extending past \(S_m\) allows the construction of rotationally symmetric \(2\)-smeary densities.}
    \label{fig:illustration-support-rot-symm}
\end{figure}

\begin{theorem}[\textbf{Smeary support characterization under rotational symmetry}]
\label{thm:semi-tight-support-rot-symm}
\label{thm:ball-obstruction-rot-symm}
\label{thm:engineering-density-in-thin-annuli}
Let \(\mu\) be an absolutely continuous probability measure on \(\mathbb S^m\), \(m\ge2\), with rotationally symmetric density about its unique local Fréchet mean \(\mubar\). Assume Assumption~\ref{assum:density-assumption}. Then the following hold.

\begin{enumerate}
\item[(i)] If \(m=2\) or \(m=3\), then \(\mu\) is not smeary at $\mubar$. More precisely:
\begin{enumerate}
\item[(a)] if
\[
\operatorname{supp}(\mu)\subset \{x\in \mathbb S^m:d(\mubar,x)\le R_m\},
\]
then \(D^2\widetilde F_\mu(0)\) is strictly positive definite;

\item[(b)] regardless of the radial support,
\[
D^4\widetilde F_\mu(0)[w,w,w,w]<0
\qquad\text{for every }w\neq0.
\]
Consequently, if \(D^2\widetilde F_\mu(0)=0\), then \(0\) is a strict local maximum of \(\widetilde F_\mu\). In particular, \(\mubar\) cannot be a smeary local Fréchet mean.
\end{enumerate}

\item[(ii)] Let \(m\ge4\), and let \(S_m\in(\pi/2,\pi)\) be the first zero of \(h_m\). If
\[
\operatorname{supp}(\mu)\subset \{x\in \mathbb S^m:d(\mubar,x)\le S_m\},
\]
then \(\mu\) is not smeary. More precisely:
\begin{enumerate}
\item[(a)] if
\[
\operatorname{supp}(\mu)\subset \{x\in \mathbb S^m:d(\mubar,x)\le R_m\},
\]
then \(D^2\widetilde F_\mu(0)\) is strictly positive definite;

\item[(b)] if
\[
\operatorname{supp}(\mu)\subset \{x\in \mathbb S^m:d(\mubar,x)\le S_m\}
\]
and \(D^2\widetilde F_\mu(0)=0\), then
\[
D^4\widetilde F_\mu(0)[w,w,w,w]<0
\qquad\text{for every }w\neq0.
\]
Thus \(0\) is a strict local maximum of \(\widetilde F_\mu\), and \(\mubar\) cannot be a smeary local Fréchet mean.
\end{enumerate}

\item[(iii)] Conversely, let \(m\ge4\). For every \(\mubar\in\mathbb S^m\) and every \(\varepsilon>0\) with \(S_m+\varepsilon<\pi\), there exists an absolutely continuous probability measure \(\nu\) on \(\mathbb S^m\), rotationally symmetric about \(\mubar\), with continuous radial density satisfying Assumption~\ref{assum:density-assumption}, such that
\[
\operatorname{supp}(\nu)\subset B(\mubar;S_m+\varepsilon),
\]
and
\[
\nu\bigl(\{x\in\mathbb S^m:S_m<d(\mubar,x)<S_m+\varepsilon\}\bigr)>0.
\]
Moreover,
\[
D^2\widetilde F_\nu(0)=0,
\qquad
D^4\widetilde F_\nu(0)[w,w,w,w]>0
\quad\text{for every }w\neq0.
\]
In particular, \(\mubar\) is a \(2\)-smeary local Fréchet mean of \(\nu\).

Consequently, since \(S_m\to\pi/2\) as \(m\to\infty\), for every \(\varepsilon_0>0\) and all sufficiently large \(m\), there exists an absolutely continuous rotationally symmetric \(2\)-smeary probability measure on \(\mathbb S^m\), satisfying Assumption~\ref{assum:density-assumption}, whose support is contained in \(B(\mubar;\pi/2+\varepsilon_0)\).
\end{enumerate}
\end{theorem}

\begin{proof}
By applying an isometry, we may assume throughout the proof that \(\mubar=N\). We write points away from the cut locus of \(N\) in polar normal coordinates as \(x=\exp_N(R\Theta)\), where \(R=d(N,x)\in(0,\pi)\) and \(\Theta\in\mathbb S^{m-1}\subset T_N\mathbb S^m\). Let \(\rho(R)\) denote the radial density.

\emph{Proof of (i)(a) and (ii)(a): the Hessian obstruction.}
By the rotationally symmetric Hessian formula \eqref{eqn:Hess},
\[
D^2\widetilde F_\mu(0)
=
\frac{2|\mathbb S^{m-1}|}{m}
\left(\int_0^\pi \rho(R)(\sin R)^{m-1}b_m(R)\,dR\right)I_m .
\]
Thus the Hessian is a scalar multiple of the identity. If
\[
\operatorname{supp}(\mu)\subset \{x\in\mathbb S^m:d(N,x)\le R_m\},
\]
then the radial density is supported in \([0,R_m]\). Proposition~\ref{prop:zero_bm} gives \(b_m(R)>0\) on \((0,R_m)\) and \(b_m(R_m)=0\). Since \(\mu\) is a probability measure with density, it cannot be supported only on the single geodesic sphere \(\{R=R_m\}\). Hence
\[
\int_0^\pi \rho(R)(\sin R)^{m-1}b_m(R)\,dR
=
\int_0^{R_m}\rho(R)(\sin R)^{m-1}b_m(R)\,dR>0.
\]
Therefore \(D^2\widetilde F_\mu(0)\) is strictly positive definite. This proves parts (i)(a) and (ii)(a). In this support regime the lifted Frechet function has nondegenerate quadratic growth at \(0\), so the Frechet mean is non-smeary.

\emph{Proof of (i)(b): the low-dimensional quartic obstruction.}
Assume \(m=2\) or \(m=3\). By Proposition~\ref{prop:sign_h_m_2_3_and_zero_m_ge_4}, \(h_m(R)<0\) for every \(R\in(0,\pi)\). Using the rotationally symmetric fourth-order formula \eqref{eqn:fourth-derivative}, we get
\[
D^4\widetilde F_\mu(0)[w,w,w,w]
=
24|\mathbb S^{m-1}|
\left(\int_0^\pi \rho(R)(\sin R)^{m-1}h_m(R)\,dR\right)\|w\|^4 .
\]
The integral is strictly negative, because \(h_m<0\) on \((0,\pi)\) and \(\mu\) is an absolutely continuous probability measure. Hence
\[
D^4\widetilde F_\mu(0)[w,w,w,w]<0
\qquad\text{for every }w\neq0.
\]

It remains to explain why this rules out smeariness. Since \(D^2\widetilde F_\mu(0)\) is a scalar multiple of \(I_m\) by \eqref{eqn:Hess}, and since \(0\) is a local minimum of \(\widetilde F_\mu\), this scalar is nonnegative. If it is strictly positive, then the mean is non-smeary. If \(D^2\widetilde F_\mu(0)=0\), then rotational symmetry removes the cubic term: the lifted Frechet function is \(O(m)\)-invariant, hence its Taylor expansion at \(0\) contains no nonzero homogeneous term of odd degree. Therefore
\[
\widetilde F_\mu(u)-\widetilde F_\mu(0)
=
\frac{1}{24}D^4\widetilde F_\mu(0)[u,u,u,u]+o(\|u\|^4).
\]
The quartic term is strictly negative for \(u\neq0\), so \(0\) is a strict local maximum of \(\widetilde F_\mu\), not a smeary local minimum. This proves part (i).

\emph{Proof of (ii)(b): the quartic obstruction up to \(S_m\).}
Let \(m\ge4\), and assume
\[
\operatorname{supp}(\mu)\subset \{x\in\mathbb S^m:d(N,x)\le S_m\}.
\]
As above, \(D^2\widetilde F_\mu(0)\) is a scalar multiple of the identity. Since \(0\) is a local minimum, this scalar is nonnegative. If it is strictly positive, then \(\mu\) is non-smeary. It remains to consider the degenerate Hessian case \(D^2\widetilde F_\mu(0)=0\).

By Lemma~\ref{lem:Rm-Sm-ordering}, \(h_m(R)<0\) for all \(0<R<S_m\), while \(h_m(S_m)=0\). Since \(\mu\) has a density, it cannot be supported only on the single geodesic sphere \(\{R=S_m\}\). Hence
\[
\int_0^\pi \rho(R)(\sin R)^{m-1}h_m(R)\,dR
=
\int_0^{S_m}\rho(R)(\sin R)^{m-1}h_m(R)\,dR<0.
\]
By \eqref{eqn:fourth-derivative},
\[
D^4\widetilde F_\mu(0)[w,w,w,w]<0
\qquad\text{for every }w\neq0.
\]
Again the cubic term vanishes by rotational symmetry, and therefore
\[
\widetilde F_\mu(u)-\widetilde F_\mu(0)
=
\frac{1}{24}D^4\widetilde F_\mu(0)[u,u,u,u]+o(\|u\|^4).
\]
Thus, in the only possible Hessian-degenerate case, \(0\) is a strict local maximum of \(\widetilde F_\mu\). Hence \(\mubar=N\) cannot be a smeary local Fréchet mean. This proves part (ii).

\emph{Proof of (iii): construction of a rotationally symmetric \(2\)-smeary density past \(S_m\).}
Fix \(m\ge4\) and \(\varepsilon>0\) with \(S_m+\varepsilon<\pi\). By Lemma~\ref{lem:positivity-h_m-past-first-zero}, choose \(\eta_m\in(0,\pi-S_m)\) such that
\[
h_m(R)>0
\qquad\text{for all }R\in(S_m,S_m+\eta_m).
\]
Set \(\eta_\varepsilon:=\min\{\eta_m,\varepsilon\}\), and choose
\[
R_+\in(S_m,S_m+\eta_\varepsilon).
\]
Since \(R_m<S_m<R_+\) by Lemma~\ref{lem:Rm-Sm-ordering}, and since \(b_m\) is strictly decreasing with unique zero \(R_m\) by Proposition~\ref{prop:zero_bm}, we have \(b_m(R_+)<0\). Also \(h_m(R_+)>0\). Define
\[
q_+:=\frac{h_m(R_+)}{-b_m(R_+)}>0.
\]

By continuity, choose a nonnegative outer bump function
\[
\phi_+\in C_c((S_m,S_m+\eta_\varepsilon)),
\qquad
\phi_+\not\equiv0,
\]
with support sufficiently concentrated near \(R_+\) so that, with
\[
I_+:=\int_{S_m}^{S_m+\eta_\varepsilon}\phi_+(R)b_m(R)\,dR,
\qquad
J_+:=\int_{S_m}^{S_m+\eta_\varepsilon}\phi_+(R)h_m(R)\,dR,
\]
we have
\[
I_+<0,\qquad J_+>0,\qquad \frac{J_+}{-I_+}>\frac{q_+}{2}.
\]
This is possible because the ratio \(J_+/(-I_+)\) approaches \(h_m(R_+)/(-b_m(R_+))=q_+\) as the support of \(\phi_+\) is concentrated near \(R_+\).

We next choose the inner bump near \(N\). Since
\[
b_m(R)=1+(m-1)R\cot R\longrightarrow m>0
\qquad\text{as }R\downarrow0,
\]
and since \(h_m(R)\to0\) as \(R\downarrow0\), we have \(|h_m(R)|/b_m(R)\to0\). Therefore, after choosing \(\delta>0\) sufficiently small, we have \(b_m(R)>0\) and
\[
\frac{|h_m(R)|}{b_m(R)}<\frac{q_+}{4}
\qquad\text{for all }R\in(0,\delta).
\]
Choose a nonnegative bump function
\[
\phi_-\in C_c((0,\delta)),
\qquad
\phi_-\not\equiv0.
\]
Define
\[
I_-:=\int_0^\delta \phi_-(R)b_m(R)\,dR>0,
\qquad
J_-:=\int_0^\delta \phi_-(R)h_m(R)\,dR.
\]
The choice of \(\delta\) gives
\[
\frac{|J_-|}{I_-}<\frac{q_+}{4}.
\]

Now define
\[
\psi_m(R):=a\,\phi_-(R)+c\,\phi_+(R),
\qquad a,c>0,
\]
and choose the ratio of the coefficients by
\[
\frac{c}{a}:=\frac{I_-}{-I_+}.
\]
Then
\begin{equation}
\begin{aligned}
\int_0^\pi \psi_m(R)b_m(R)\,dR
&=
a\int_0^\delta \phi_-(R)b_m(R)\,dR
+
c\int_{S_m}^{S_m+\eta_\varepsilon}\phi_+(R)b_m(R)\,dR \\
&=
aI_-+cI_+
=
0.
\end{aligned}
\end{equation}
Moreover,
\begin{align}
\int_0^\pi \psi_m(R)h_m(R)\,dR
&=
a\int_0^\delta \phi_-(R)h_m(R)\,dR
+
c\int_{S_m}^{S_m+\eta_\varepsilon}\phi_+(R)h_m(R)\,dR \notag\\
&=
aJ_-+cJ_+ \notag\\
&=
aI_-\left(\frac{J_-}{I_-}+\frac{J_+}{-I_+}\right).
\end{align}
Using the two ratio estimates,
\[
\frac{J_-}{I_-}+\frac{J_+}{-I_+}
>
-\frac{q_+}{4}+\frac{q_+}{2}
=
\frac{q_+}{4}>0.
\]
Thus
\[
\int_0^\pi \psi_m(R)h_m(R)\,dR>0.
\]

We now turn \(\psi_m\) into a radial density. Define
\[
\rho(R):=
C_m\,\frac{\psi_m(R)}{(\sin R)^{m-1}},
\]
where \(C_m>0\) is chosen so that
\[
|\mathbb S^{m-1}|
\int_0^\pi \rho(R)(\sin R)^{m-1}\,dR
=
1.
\]
Since \(\phi_-\) and \(\phi_+\) are compactly supported in \((0,\pi)\), the function \(\rho\) is continuous and vanishes in a neighborhood of the cut locus. Hence \(\rho\) defines an absolutely continuous rotationally symmetric probability measure \(\nu\) satisfying Assumption~\ref{assum:density-assumption}. Its radial support is contained in
\[
(0,\delta)\cup(S_m,S_m+\eta_\varepsilon)
\subset(0,S_m+\varepsilon),
\]
and therefore
\[
\operatorname{supp}(\nu)\subset B(N;S_m+\varepsilon).
\]
Moreover, since \(\phi_+\not\equiv0\),
\[
\nu\bigl(\{x\in\mathbb S^m:S_m<d(N,x)<S_m+\varepsilon\}\bigr)>0.
\]

Finally, we verify the derivative conditions. By \eqref{eqn:Hess},
\[
D^2\widetilde F_\nu(0)
=
\frac{2|\mathbb S^{m-1}|}{m}
\left(
\int_0^\pi \rho(R)(\sin R)^{m-1}b_m(R)\,dR
\right)I_m.
\]
Since \(\rho(R)(\sin R)^{m-1}=C_m\psi_m(R)\), the integral in parentheses is zero by the construction of \(\psi_m\). Thus \(D^2\widetilde F_\nu(0)=0\).

Similarly, by \eqref{eqn:fourth-derivative},
\[
D^4\widetilde F_\nu(0)[w,w,w,w]
=
24|\mathbb S^{m-1}|
\left(
\int_0^\pi \rho(R)(\sin R)^{m-1}h_m(R)\,dR
\right)
\|w\|^4.
\]
The scalar integral equals
\[
C_m\int_0^\pi\psi_m(R)h_m(R)\,dR,
\]
which is strictly positive by the construction above. Hence
\[
D^4\widetilde F_\nu(0)[w,w,w,w]>0
\qquad\text{for every }w\neq0.
\]
Rotational symmetry gives \(D\widetilde F_\nu(0)=0\) and removes all odd-order terms in the Taylor expansion at \(0\). Therefore
\[
\widetilde F_\nu(u)-\widetilde F_\nu(0)
=
\frac{1}{24}D^4\widetilde F_\nu(0)[u,u,u,u]+o(\|u\|^4).
\]
Thus \(\widetilde F_\nu(u)>\widetilde F_\nu(0)\) for all sufficiently small \(u\neq0\), and \(N\) is a \(2\)-smeary local Fréchet mean of \(\nu\). Finally, applying an isometry sending \(N\) to the prescribed center \(\mubar\) transports the construction to a rotationally symmetric density about \(\mubar\), preserving absolute continuity, the derivative conditions, and the support-radius statements. This proves part (iii).
\end{proof}

\begin{remark}[\textbf{Low-dimensional obstruction and connection to the result in \citet{Tran2021CLT}}]
The non-smeariness statement for dimensions \(2,3\) is consistent with \citet[Corollary~2.7]{Tran2021CLT}. Under the assumption that the density vanishes in a neighborhood of the cut locus, Tran proved that for \(m=2,3\), if the Hessian vanishes at a rotationally symmetric Frechet mean, then that point is a local maximum, and hence no smeary effect can occur. The proof here uses the explicit radial quartic coefficient \(h_m\): Proposition~\ref{prop:sign_h_m_2_3_and_zero_m_ge_4} gives \(h_2<0\) and \(h_3<0\) pointwise on \((0,\pi)\).
\end{remark}

\begin{remark}[\textbf{Comparison with the high-dimensional construction}]
Theorem~\ref{thm:semi-tight-support-rot-symm} is a support-threshold statement relative to the intrinsic radius \(S_m\): once \(m\ge4\) is \textit{fixed}, the construction works with support contained in \(B(\mubar;S_m+\varepsilon)\) for every \(\varepsilon>0\) with \(S_m+\varepsilon<\pi\). Thus the dimension is no longer tied to the\textit{ support leakage parameter} \(\varepsilon\) in this intrinsic-threshold formulation.

This should be distinguished from \textit{hemispherical formulations}, such as \citet[Theorem~3.3(i)]{Eltzner2022GeometricalSmeariness}, where the support is required to lie within a radius of the form \(\pi/2+O(1/m)\). To force the support into a prescribed ball \(B(\mubar;\pi/2+\delta)\), one must then use the asymptotic behavior \(S_m \to\pi/2\) as \(m\to\infty\), see Proposition \ref{prop:asymptotics-Rm-Sm}. This is precisely the final consequence in part~(iii).
\end{remark}

\section{Sharp support threshold for general densities}
\label{scn:general-density-support-threshold}

We now turn to absolutely continuous probability measures with densities not assumed to be rotationally symmetric. The setup and the derivative formulas for the Fréchet function have already been established in Section~\ref{scn:common-setup-derivatives}. In particular, we use the general Hessian formula \eqref{eqn:Hessian-general-density}, the bilinear form \eqref{eqn:Hessian-general-bilinear-form}, the fourth derivative formula \eqref{eqn:fourth-derivative-general-density}, and the averaged fourth derivative identity \eqref{eqn:averaged-fourth-derivative-general-density}. The theorem below collects the corresponding support obstructions and the sharp directional construction. In the theorem below, we apply an isometry to assume $\mubar:=N,$ the North pole without loss of generality.

\begin{theorem}[\textbf{Support threshold for general densities}]
\label{thm:general-density-support-threshold}
\label{thm:hemisphre-obstruction-general}
\label{thm:directionally-smeary-rank-one}
\label{prop:rank1_psd_hessian_ball_pihalf_eps}
\label{prop:fourth_derivative_positive_definite}
Let \(\mu\) be an absolutely continuous probability measure on \(\mathbb S^m\), \(m\ge2\), satisfying Assumption~\ref{assum:density-assumption}, and suppose that \(N\) is a unique local Fréchet mean of \(\mu\). Let \(\widetilde F_\mu\) be the lifted Fréchet function in normal coordinates at \(N\). Then the following hold.

\begin{enumerate}
\item[(i)] If
\[
\operatorname{supp}(\mu)\subset \overline B_g(N;\pi/2),
\]
then \(D^2\widetilde F_\mu(0)\) is positive definite. Hence neither full nor directional smeariness can occur at \(N\).

\item[(ii)] Let \(R_m\in(\pi/2,\pi)\) be the unique zero of \(b_m\) (cf. Equation \eqref{eqn:R_m}). If
\[
\operatorname{supp}(\mu)\subset \overline B_g(N;R_m),
\]
then \(D^2\widetilde F_\mu(0)\neq0\). Hence full smeariness is impossible.

\item[(iii)] Conversely, for every \(m\ge2\) and every \(\varepsilon>0\), there exists an absolutely continuous probability measure \(\nu\) on \(\mathbb S^m\), satisfying Assumption~\ref{assum:density-assumption}, such that
\[
\operatorname{supp}(\nu)\subset B_g(N;\pi/2+\varepsilon),
\]
\(N\) is a local Fréchet mean of \(\nu\), \(D^2\widetilde F_\nu(0)\) is positive semidefinite of rank one, and \(\nu\) is directionally \(2\)-smeary at \(N\). More precisely, after choosing an orthonormal basis \(e_1,\ldots,e_m\) of \(T_N\mathbb S^m\), one may arrange that
\[
\ker D^2\widetilde F_\nu(0)=e_1^\perp,
\qquad
D^3\widetilde F_\nu(0)=0,
\]
and
\[
D^4\widetilde F_\nu(0)[w,w,w,w]>0
\qquad\text{for every }w\in e_1^\perp\setminus\{0\}.
\]
\item[(iv)] Let \(m\ge4\), and let \(S_m\in(\pi/2,\pi)\) be the first zero of \(h_m\) (cf. Equation \eqref{eqn:S_m}). If
\[
\operatorname{supp}(\mu)\subset \overline B_g(N;S_m),
\]
then full smeariness at \(N\) is impossible. More precisely, if
\(D^2\widetilde F_\mu(0)=0\), then there exists \(w\in T_N\mathbb S^m\setminus\{0\}\) such that
\[
D^4\widetilde F_\mu(0)[w,w,w,w]<0.
\]
Consequently, \(0\) cannot be a local minimum of \(\widetilde F_\mu\), and hence \(N\) cannot be a fully smeary local Fréchet mean.

\item[(v)] Conversely, let \(m\ge4\). For every \(\varepsilon>0\) with \(S_m+\varepsilon<\pi\), there exists an absolutely continuous probability measure \(\nu\) on \(\mathbb S^m\), satisfying Assumption~\ref{assum:density-assumption}, such that
\[
\operatorname{supp}(\nu)\subset B_g(N;S_m+\varepsilon),
\]
and \(N\) is a fully \(2\)-smeary local Fréchet mean of \(\nu\). Moreover, \(\nu\) can be chosen rotationally symmetric about \(N\).
\end{enumerate}
\end{theorem}

\begin{proof}
Write points away from the cut locus of \(N\) as \(x=\exp_N(R\Theta)\), where \(R\in(0,\pi)\) and \(\Theta\in\mathbb S^{m-1}\subset T_N\mathbb S^m\). Let \(H:=D^2\widetilde F_\mu(0)\).

\emph{Proof of (i).}
Let \(w\in T_N\mathbb S^m\). By \eqref{eqn:Hessian-general-bilinear-form},
\[
w^\top H w
=
2\int_0^\pi\!\!\int_{\mathbb S^{m-1}}
\left[
\frac{R}{\tan R}\|w\|^2+
\left(1-\frac{R}{\tan R}\right)\langle w,\Theta\rangle^2
\right]\rho(R,\Theta)(\sin R)^{m-1}\,d\Theta\,dR .
\]
If the support is contained in \(\overline B_g(N;\pi/2)\), then \(R/\tan R\ge0\) and \(1-R/\tan R\ge0\) on the support. Thus \(w^\top H w\ge0\). Moreover, since \(\mu\) is absolutely continuous and is a probability measure, it gives positive mass to the region \(R<\pi/2\). On that region \(R/\tan R>0\), and hence the integral is strictly positive for every \(w\neq0\). Thus \(H\) is positive definite, which rules out both full and directional smeariness.

\emph{Proof of (ii).}
Taking the trace in \eqref{eqn:Hessian-general-density} gives
\[
\operatorname{tr}(H)
=
2\int_0^\pi\!\!\int_{\mathbb S^{m-1}}
b_m(R)\rho(R,\Theta)(\sin R)^{m-1}\,d\Theta\,dR .
\]
By Proposition~\ref{prop:zero_bm}, \(b_m>0\) on \((0,R_m)\) and \(b_m(R_m)=0\). If \(\operatorname{supp}(\mu)\subset \overline B_g(N;R_m)\), then the integrand is nonnegative, and it is positive on a set of positive measure because \(\mu\) is absolutely continuous and not supported only on the geodesic sphere \(R=R_m\). Hence \(\operatorname{tr}(H)>0\), so \(H\neq0\). Therefore full smeariness, which requires the Hessian $H$ to vanish, is impossible.

\emph{Proof of (iii).}
Fix \(m\ge2\) and \(\varepsilon>0\). Set \(a(R):=R/\tan R\) and \(w_m(R):=(\sin R)^{m-1}\). We first choose an even angular density concentrated near the two points \(\pm e_1\). For \(\kappa>0\), put
\[
\varphi_\kappa(\Theta):=\frac{\exp(\kappa\Theta_1^2)}
{\int_{\mathbb S^{m-1}}\exp(\kappa\Theta_1^2)\,d\Theta}.
\]
Then \(\varphi_\kappa\) is continuous, positive, even, and invariant under rotations preserving the \(e_1\)-axis. Its second moment matrix is diagonal:
\[
\Sigma_\kappa
:=
\int_{\mathbb S^{m-1}}\Theta\Theta^\top\varphi_\kappa(\Theta)\,d\Theta
=
\operatorname{diag}(\lambda_\parallel,\lambda_\perp,\ldots,\lambda_\perp),
\]
with \(\lambda_\parallel+(m-1)\lambda_\perp=1\), \(\lambda_\parallel\to1\), and \(\lambda_\perp\to0\) as \(\kappa\to\infty\). This follows from the symmetries of \(\varphi_\kappa\) and its concentration near \(\{\pm e_1\}\).

Choose \(\delta\in(0,\varepsilon)\) so small that \([\pi/2-\delta,\pi/2+\delta]\subset(0,\pi)\). Choose nonnegative smooth radial bumps \(g_+\) and \(g_-\) such that
\[
\operatorname{supp}(g_+)\subset \left(\frac{\pi}{2}-\delta,\frac{\pi}{2}-\frac{\delta}{2}\right),
\qquad
\operatorname{supp}(g_-)\subset \left(\frac{\pi}{2}+\frac{\delta}{2},\frac{\pi}{2}+\delta\right),
\]
and \(\int_0^\pi g_\pm(R)w_m(R)\,dR=1\). For \(p\in[0,1]\), set \(g_p:=pg_++(1-p)g_-\), and define
\[
\rho_p(R,\Theta):=g_p(R)\varphi_\kappa(\Theta).
\]
This defines an absolutely continuous probability measure \(\nu_p\), supported in \(B_g(N;\pi/2+\varepsilon)\), and satisfying Assumption~\ref{assum:density-assumption}. Since \(\varphi_\kappa\) is even, \(\int_{\mathbb S^{m-1}}\Theta\varphi_\kappa(\Theta)\,d\Theta=0\), so \(D\widetilde F_{\nu_p}(0)=0\).

Let
\[
s(p):=\int_0^\pi a(R)g_p(R)w_m(R)\,dR.
\]
Using \eqref{eqn:Hessian-general-density}, the Hessian is
\[
D^2\widetilde F_{\nu_p}(0)=2\{s(p)I_m+(1-s(p))\Sigma_\kappa\}.
\]
Hence its eigenvalues are \(2\mu_\parallel(p)\) in the \(e_1\)-direction and \(2\mu_\perp(p)\) on \(e_1^\perp\), where
\[
\mu_\parallel(p)=s(p)+(1-s(p))\lambda_\parallel,
\qquad
\mu_\perp(p)=\lambda_\perp+s(p)(1-\lambda_\perp).
\]
Since \(a(R)>0\) below \(\pi/2\) and \(a(R)<0\) above \(\pi/2\), we have \(s(1)>0\) and \(s(0)<0\). Taking \(\kappa\) sufficiently large, we may assume \(\lambda_\perp<-s(0)/(1-s(0))\). Then \(\mu_\perp(0)<0\), while \(\mu_\perp(1)>0\). By continuity there exists \(p^\ast\in(0,1)\) such that \(\mu_\perp(p^\ast)=0\). For this \(p^\ast\),
\[
s(p^\ast)=-\frac{\lambda_\perp}{1-\lambda_\perp},
\qquad
\mu_\parallel(p^\ast)
=
\frac{\lambda_\parallel-\lambda_\perp}{1-\lambda_\perp}>0.
\]
Thus \(D^2\widetilde F_{\nu_{p^\ast}}(0)\) is positive semidefinite of rank one and has kernel \(V:=e_1^\perp\).

We now verify the higher-order behavior on \(V\). Since the angular density is even, the cubic term in \eqref{eqn:g-eqn} integrates to zero; equivalently, \(D^3\widetilde F_{\nu_{p^\ast}}(0)=0\). For \(w\in V\), the fourth derivative formula \eqref{eqn:fourth-derivative-general-density} gives
\[
D^4\widetilde F_{\nu_{p^\ast}}(0)[w,w,w,w]
=
24\|w\|^4 I,
\]
where
\[
I
:=
\int_0^\pi
\left(A_0(R)+\lambda_\perp A_2(R)R^2+\beta_\perp A_4(R)R^4\right)
g_{p^\ast}(R)w_m(R)\,dR,
\]
and \(\beta_\perp:=\int_{\mathbb S^{m-1}}\Theta_2^4\varphi_\kappa(\Theta)\,d\Theta\). Here we used that, for \(w\in e_1^\perp\),
\[
\int_{\mathbb S^{m-1}}\langle w,\Theta\rangle^2\varphi_\kappa(\Theta)\,d\Theta=\lambda_\perp\|w\|^2,
\qquad
\int_{\mathbb S^{m-1}}\langle w,\Theta\rangle^4\varphi_\kappa(\Theta)\,d\Theta=\beta_\perp\|w\|^4.
\]

It remains to show \(I>0\). Write \(t=R-\pi/2\). A direct expansion of the coefficients in \eqref{eqn:A-B} gives
\[
A_0\left(\frac{\pi}{2}+t\right)=\frac{\pi}{24}t+\frac13t^2+O(t^3),
\qquad
a\left(\frac{\pi}{2}+t\right)=-\frac{\pi}{2}t-t^2+O(t^3).
\]
Therefore
\[
A_0\left(\frac{\pi}{2}+t\right)+\frac1{12}a\left(\frac{\pi}{2}+t\right)
=
\frac14t^2+O(t^3).
\]
After shrinking \(\delta\) if necessary, this gives
\[
A_0(R)\ge -\frac1{12}a(R)+\frac18\left(R-\frac{\pi}{2}\right)^2
\quad\text{on }\left[\frac{\pi}{2}-\delta,\frac{\pi}{2}+\delta\right].
\]
Also \(A_2(R)R^2\) and \(A_4(R)R^4\) are bounded on this interval; say
\[
|A_2(R)R^2|+|A_4(R)R^4|\le M .
\]
Since \(\mu_\perp(p^\ast)=0\), we have
\[
\int_0^\pi a(R)g_{p^\ast}(R)w_m(R)\,dR
=
-\frac{\lambda_\perp}{1-\lambda_\perp}.
\]
Using the support choice of \(g_+\) and \(g_-\), we also have
\[
\int_0^\pi\left(R-\frac{\pi}{2}\right)^2g_{p^\ast}(R)w_m(R)\,dR
\ge \frac{\delta^2}{4}.
\]
Consequently,
\[
I
\ge
\frac{\delta^2}{32}-M\lambda_\perp-M\beta_\perp.
\]
Finally, \(0\le\beta_\perp\le\lambda_\perp\), and \(\lambda_\perp\to0\) as \(\kappa\to\infty\). Taking \(\kappa\) still larger if necessary, we obtain \(I>0\). Hence
\[
D^4\widetilde F_{\nu_{p^\ast}}(0)[w,w,w,w]>0
\qquad\text{for every }w\in V\setminus\{0\}.
\]

Thus \(\widetilde F_{\nu_{p^\ast}}\) has zero first derivative, positive semidefinite rank-one Hessian, vanishing third derivative, and strictly positive fourth derivative along the Hessian kernel \(V=e_1^\perp\). It follows from the Taylor expansion that \(N\) is a local Fréchet mean and is directionally \(2\)-smeary along \(V\). Setting \(\nu:=\nu_{p^\ast}\) completes the proof.

\emph{Proof of (iv): the \(S_m\)-obstruction to full smeariness.}
Assume \(m\ge4\) and
\[
\operatorname{supp}(\mu)\subset \overline B_g(N;S_m).
\]
If \(D^2\widetilde F_\mu(0)\neq0\), then full smeariness is already impossible. Hence suppose that \(D^2\widetilde F_\mu(0)=0\). By the averaged fourth derivative identity \eqref{eqn:averaged-fourth-derivative-general-density},
\[
\int_{\mathbb S^{m-1}}
D^4\widetilde F_\mu(0)[w,w,w,w]\,d\sigma(w)
=
24\int_0^\pi\int_{\mathbb S^{m-1}}
h_m(R)\rho(R,\Theta)(\sin R)^{m-1}\,d\Theta\,dR,
\]
where \(d\sigma\) denotes normalized surface measure on the unit sphere in \(T_N\mathbb S^m\). By Lemma~\ref{lem:Rm-Sm-ordering}, \(h_m(R)<0\) for \(0<R<S_m\), while \(h_m(S_m)=0\). Since \(\mu\) is absolutely continuous, it gives no mass to the geodesic spheres \(R=0\) and \(R=S_m\). Therefore the right-hand side is strictly negative. Hence there exists \(w\in\mathbb S^{m-1}\) such that
\[
D^4\widetilde F_\mu(0)[w,w,w,w]<0.
\]
If \(0\) were a local minimum of \(\widetilde F_\mu\), then for \(\varphi(t):=\widetilde F_\mu(tw)\) we would have \(\varphi'(0)=0\) and \(\varphi''(0)=0\). If \(\varphi^{(3)}(0)\neq0\), then one side of \(0\) gives \(\varphi(t)<\varphi(0)\) for small \(t\). If \(\varphi^{(3)}(0)=0\), the negative fourth derivative gives the same conclusion. Thus \(0\) cannot be a local minimum. Hence full smeariness at \(N\) is impossible.

\emph{Proof of (v): sharpness beyond \(S_m\).}
This follows directly from the rotationally symmetric construction in Theorem~\ref{thm:semi-tight-support-rot-symm}(iii). Indeed, rotationally symmetric densities are a special case of general densities. Applying that theorem with center \(N\), for every \(m\ge4\) and every \(\varepsilon>0\) with \(S_m+\varepsilon<\pi\), we obtain an absolutely continuous probability measure \(\nu\), satisfying Assumption~\ref{assum:density-assumption}, such that
\[
\operatorname{supp}(\nu)\subset B_g(N;S_m+\varepsilon),
\]
and \(N\) is a fully \(2\)-smeary local Fréchet mean of \(\nu\). Moreover, the construction in Theorem~\ref{thm:semi-tight-support-rot-symm}(iii) is rotationally symmetric about \(N\), as claimed.

\end{proof}

\section{Classical CLT consequences below the support threshold}
\label{scn:bpclt-consequences}

We record the corresponding classical central limit theorem in the support regimes where the results above force the Hessian of the Fr\'echet function to be positive definite. This is an application of the Bhattacharya--Patrangenaru CLT, not a new CLT. The point is that the support conditions identify regimes in which Hessian-degenerate, and hence potentially smeary, behavior is ruled out.

Let \(X\sim\mu\) be an \(\mathbb S^m\)-valued random variable with unique population Fr\'echet mean \(N\), and let \(\hat\mu_n\) be a measurable selection of sample Fr\'echet means. Set
\[
    U_n:=\log_N(\hat\mu_n),
    \qquad
    \Sigma:=\operatorname{Cov}(\log_N X),
    \qquad
    H:=D^2\widetilde F_\mu(0).
\]
Assume the usual regularity and consistency hypotheses of the Bhattacharya--Patrangenaru CLT; see \citet{BP,BP2}. In the present setting, Assumption~\ref{assum:density-assumption} gives the required smoothness of the lifted Fr\'echet function near \(0\in T_N\mathbb S^m\).

\begin{corollary}[\textbf{BPCLT below the support thresholds}]
\label{cor:bpclt-subthreshold}
The following hold.

\begin{enumerate}
\item[\textup{(i)}]
Assume that \(\mu\) has a rotationally symmetric density \(\rho(R)\) about \(N\), satisfies Assumption~\ref{assum:density-assumption}, and
\[
    \operatorname{supp}(\mu)\subset \overline B_g(N;R_*)
    \qquad\text{for some }R_*\le R_m .
\]
Then \(H=\lambda_m I_m\), where
\[
    \lambda_m
    :=
    \frac{2\operatorname{vol}(\mathbb S^{m-1})}{m}
    \int_0^{R_*}
    b_m(R)\rho(R)(\sin R)^{m-1}\,dR
    >0.
\]
Moreover,
\[
    \Sigma=\sigma_m^2 I_m,
    \qquad
    \sigma_m^2
    :=
    \frac{\operatorname{vol}(\mathbb S^{m-1})}{m}
    \int_0^{R_*}
    R^2\rho(R)(\sin R)^{m-1}\,dR .
\]
Consequently,
\[
    \sqrt n\,U_n
    \xrightarrow{\mathcal D}
    \mathcal N\!\left(0,\frac{4\sigma_m^2}{\lambda_m^2}I_m\right).
\]
If \(\sigma_m^2>0\), then
\[
    \frac{n\lambda_m^2}{4\sigma_m^2}\|U_n\|^2
    \xrightarrow{\mathcal D}
    \chi_m^2 .
\]

\item[\textup{(ii)}]
Assume that \(\mu\) is absolutely continuous, satisfies Assumption~\ref{assum:density-assumption}, is not necessarily rotationally symmetric, and
\[
    \operatorname{supp}(\mu)\subset \overline B_g(N;\pi/2).
\]
Then Theorem~\ref{thm:general-density-support-threshold} gives that \(H\) is positive definite. Hence
\[
    \sqrt n\,U_n
    \xrightarrow{\mathcal D}
    \mathcal N\!\left(0,4H^{-1}\Sigma H^{-1}\right),
\]
where \(H\) is the explicit matrix in \eqref{eqn:Hessian-general-density}, equivalently the bilinear form in \eqref{eqn:Hessian-general-bilinear-form}. If \(\Sigma\) is positive definite, then
\[
    \frac n4\,U_n^\top H\Sigma^{-1}HU_n
    \xrightarrow{\mathcal D}
    \chi_m^2 .
\]
\end{enumerate}
\end{corollary}

\begin{proof}
We first prove \textup{(i)}. By Theorem~\ref{thm:semi-tight-support-rot-symm}, or directly from \eqref{eqn:Hess}, the support condition \(\operatorname{supp}(\mu)\subset\overline B_g(N;R_*)\) with \(R_*\le R_m\) gives
\[
    H
    =
    \frac{2\operatorname{vol}(\mathbb S^{m-1})}{m}
    \left(
    \int_0^{R_*}
    b_m(R)\rho(R)(\sin R)^{m-1}\,dR
    \right)I_m
    =
    \lambda_m I_m,
\]
with \(\lambda_m>0\). Since \(\log_N X=R\Theta\) and the density is rotationally symmetric,
\[
    \Sigma
    =
    \int_0^{R_*}\int_{\mathbb S^{m-1}}
    R^2\Theta\Theta^\top
    \rho(R)(\sin R)^{m-1}\,d\Theta\,dR .
\]
Using
\[
    \int_{\mathbb S^{m-1}}\Theta\Theta^\top\,d\Theta
    =
    \frac{\operatorname{vol}(\mathbb S^{m-1})}{m}I_m,
\]
we get \(\Sigma=\sigma_m^2I_m\), with \(\sigma_m^2\) as above.

The Bhattacharya--Patrangenaru CLT gives
\[
    \sqrt n\,U_n
    \xrightarrow{\mathcal D}
    \mathcal N(0,4H^{-1}\Sigma H^{-1}).
\]
Substituting \(H=\lambda_m I_m\) and \(\Sigma=\sigma_m^2I_m\) yields
\[
    4H^{-1}\Sigma H^{-1}
    =
    \frac{4\sigma_m^2}{\lambda_m^2}I_m.
\]
This proves the asserted normal limit. The displayed \(\chi_m^2\)-limit follows by applying the inverse of the limiting covariance matrix.

For \textup{(ii)}, Theorem~\ref{thm:general-density-support-threshold} gives that \(H=D^2\widetilde F_\mu(0)\) is positive definite under the closed hemispherical support condition. Hence \(H^{-1}\) exists, and the Bhattacharya--Patrangenaru CLT gives
\[
    \sqrt n\,U_n
    \xrightarrow{\mathcal D}
    \mathcal N(0,4H^{-1}\Sigma H^{-1}).
\]
If \(\Sigma\) is positive definite, then the limiting covariance matrix is positive definite and
\[
    \left(4H^{-1}\Sigma H^{-1}\right)^{-1}
    =
    \frac14 H\Sigma^{-1}H.
\]
Therefore
\[
    n\,U_n^\top
    \left(4H^{-1}\Sigma H^{-1}\right)^{-1}
    U_n
    =
    \frac n4\,U_n^\top H\Sigma^{-1}HU_n
    \xrightarrow{\mathcal D}
    \chi_m^2 .
\]
This completes the proof.
\end{proof}


\section{Asymptotics of the threshold radii $R_m, S_m$}
\label{app:asymptotics-Rm-Sm}

The support thresholds in the main text are fixed-dimensional. For comparison with high-dimensional constructions, we record the order of their distance from the hemisphere.

\begin{proposition}[\textbf{High-dimensional location of \(R_m\) and \(S_m\)}]
\label{prop:asymptotics-Rm-Sm}
\label{prop:asymp-first-zero-h_m}
As \(m\to\infty\),
\[
R_m=\frac{\pi}{2}+O\!\left(\frac{1}{m}\right),
\qquad
S_m=\frac{\pi}{2}+O\!\left(\frac{1}{m}\right).
\]
In particular, \(R_m\to\pi/2\) and \(S_m\to\pi/2\).
\end{proposition}

\begin{proof}
\textbf{Analysis of \(R_m\).} By Proposition~\ref{prop:zero_bm}, \(R_m\) is the unique zero of \(b_m(R)=1+(m-1)R\cot R\), and \(R_m\in(\pi/2,\pi)\). Write
\[
R_m=\frac{\pi}{2}+\delta_m,
\qquad
n:=m-1,
\]
so that \(\delta_m\in(0,\pi/2)\). Since \(\cot(\pi/2+\delta)=-\tan\delta\), the equation \(b_m(R_m)=0\) becomes
\[
\left(\frac{\pi}{2}+\delta_m\right)\tan\delta_m=\frac{1}{n}.
\]
As \(0<\delta_m<\pi/2\), we have \(\delta_m\le\tan\delta_m\). Hence
\[
\frac{\pi}{2}\delta_m
\le
\left(\frac{\pi}{2}+\delta_m\right)\tan\delta_m
=
\frac{1}{n}.
\]
Therefore \(0<\delta_m\le 2/(\pi n)\), and hence
\[
R_m=\frac{\pi}{2}+O\!\left(\frac{1}{m}\right).
\]

\textbf{Analysis of \(S_m\).}  Set \(S_0:=\pi/2\). From \eqref{eqn:h},
\[
h_m(R)=A_0(R)+\frac{1}{m}A_2(R)R^2+\frac{3}{m(m+2)}A_4(R)R^4.
\]
Direct evaluation of the coefficients in \eqref{eqn:A-B} at \(S_0\) gives
\[
A_0(S_0)=0,
\qquad
A_0'(S_0)=\frac{\pi}{24}>0,
\qquad
A_2(S_0)S_0^2=-\frac{1}{3}.
\]
Also \(A_4(S_0)S_0^4\) is finite. Therefore
\[
h_m(S_0)
=
-\frac{1}{3m}
+
O\!\left(\frac{1}{m^2}\right)
<0
\]
for all sufficiently large \(m\).

Since \(A_0'(S_0)=\pi/24\), there exists \(r>0\) such that
\[
A_0(S_0+\delta)\ge \frac{\pi}{48}\delta
\qquad
\text{for all }0\le\delta\le r.
\]
Moreover, \(A_2(R)R^2\) and \(A_4(R)R^4\) are bounded on \([S_0,S_0+r]\). Hence there exists \(C>0\) such that, for all \(0\le\delta\le r\) and all \(m\ge4\),
\[
h_m(S_0+\delta)
\ge
\frac{\pi}{48}\delta-\frac{C}{m}.
\]
Choose \(\Delta>48C/\pi\). For all sufficiently large \(m\), we have \(\Delta/m\le r\), and therefore
\[
h_m\!\left(S_0+\frac{\Delta}{m}\right)
\ge
\frac{1}{m}\left(\frac{\pi\Delta}{48}-C\right)>0.
\]
Since \(h_m(S_0)<0\) and \(h_m\) is continuous, \(h_m\) has a zero in
\[
\left(S_0,S_0+\frac{\Delta}{m}\right).
\]
By definition, \(S_m\) is the first zero of \(h_m\) in \((\pi/2,\pi)\). Consequently,
\[
0<S_m-S_0\le \frac{\Delta}{m}
\]
for all sufficiently large \(m\). Thus
\[
S_m=\frac{\pi}{2}+O\!\left(\frac{1}{m}\right),
\]
and the proof is complete.
\end{proof}

Figure~\ref{fig:bm_hm} below shows the functions $b_m$ and $h_m$ for $m\in\{2,3,5,10,25,50\}$.

\begin{figure}[htbp]
\centering

\includegraphics[width=0.48\textwidth]{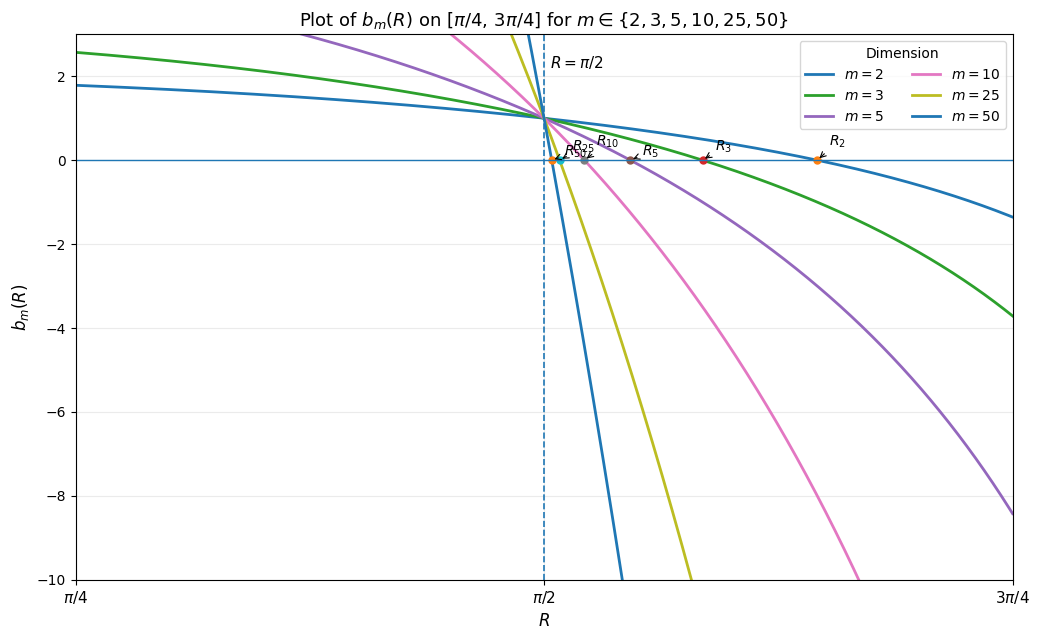}
\hfill
\includegraphics[width=0.48\textwidth]{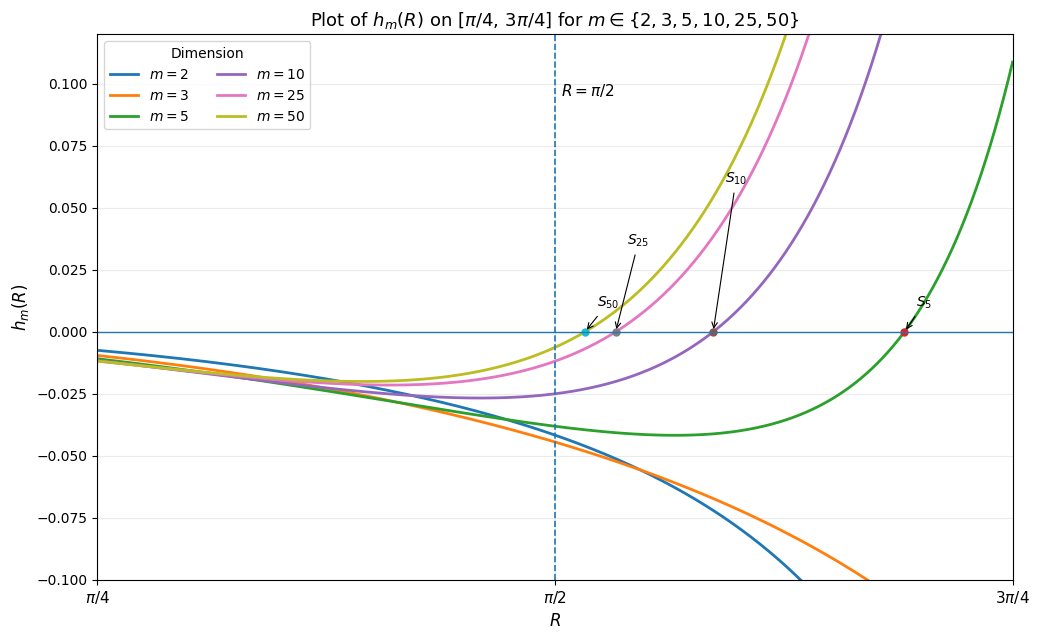}

\caption{
Plots of $b_m(R)$ (left) and $h_m(R)$ (right) as functions of $R$ on the interval $[\pi/4,\,3\pi/4]$ for $m \in \{2,3,5,10,25,50\}$. 
The dashed vertical line indicates $R=\pi/2$.}
\label{fig:bm_hm}
\end{figure}

\appendix
\section{Symbolic Computation Code}\label{scn:symbolic-computation}

The following Python (SymPy) code was used to compute the Taylor coefficients:

\begin{verbatim}
import sympy as sp

# Symbols
t, a, R, alpha = sp.symbols('t a R alpha', real=True)
r = t*a

X = (sp.cos(r)*sp.cos(R)
     + (sp.sin(r)/r)*(sp.sin(R)/R)*(t*alpha))

f = sp.acos(X)**2

series_f = sp.series(f, t, 0, 5)
series_noO = sp.expand(series_f.removeO())

# Extract coefficients
coeffs = {}
for k in range(5):
    ck = sp.simplify(series_noO.coeff(t, k))
    coeffs[k] = ck
    print(f"C{k} =", ck)
\end{verbatim}


\bibliographystyle{plainnat}   
\bibliography{references}      

\end{document}